\documentclass{jokayama}

\usepackage{color}

\newtheorem{thm}{Theorem}[section]

\theoremstyle{definition}
\newtheorem{dfn}{Definition}
\theoremstyle{remark}
\newtheorem{rmk}{Remark}

\numberwithin{equation}{section}

\newtheorem{proposition}{Proposition}[section]
\newtheorem{example}{Example}[section]
\newtheorem{corollary}[thm]{Corollary}

\newcommand{\OIn}[1]{\mathrm{O}\mathbb{I}(#1)}
\newcommand{\OIpn}[1]{\mathrm{O}\mathbb{I}_{\mathrm{p}}(#1)}
\newcommand{\rmd}{\mathrm{d}}

\usepackage{enumerate}

\allowdisplaybreaks
\begin{document}
\title[Differential geometry of isotropic invariant surfaces]{Differential geometry of invariant surfaces in simply isotropic and pseudo-isotropic spaces} 
\author[Luiz C. B. da~Silva] 
       {Luiz C. B. da Silva}
\address[Luiz C. B. da~Silva]
        {Department of Physics of Complex Systems\\
         Weizmann Institute of Science\\
         Rehovot 7610001, Israel}
\email{luiz.da-silva@weizmann.ac.il}


\begin{abstract}
We study invariant surfaces generated by one-parameter subgroups of simply and pseudo isotropic rigid motions. Basically, the simply and pseudo isotropic geometries are the study of a three-dimensional space equipped with a rank 2 metric of index zero and one, respectively. We show that the one-parameter subgroups of isotropic rigid motions lead to seven types of invariant surfaces, which then gene\-ra\-lizes the study of revolution and helicoidal surfaces in Euclidean and Lorentzian spaces to the context of singular metrics. After compu\-ting the two fundamental forms of these surfaces and their Gaussian and mean curvatures, we solve the corresponding problem of prescribed cur\-vature for invariant surfaces whose generating curves lie on a plane containing the degenerated direction. 
\end{abstract}
\maketitle
\section{Introduction}

Needless to say, there is a great interest in manifolds equipped with a metric which is not necessarily positive definite. However, less attention has been paid to cases where the metric is allowed to be degenerated. Fortunately, there has been a recent and growing interest in the geometry of spaces equipped with a degenerated metric from both applied and pure mathematical viewpoints. In this work, we study simply isotropic $\mathbb{I}^3$ and pseudo-isotropic $\mathbb{I}_{\mathrm{p}}^3$ spaces  which   are the affine space $\mathbb{R}^3$ equipped with the degenerated metric $\rmd s^2 = \rmd x^2+ \rmd y^2$ and $\rmd s^2 = \rmd x^2- \rmd y^2$, respectively. The study of $\mathbb{I}^3$ has been initiated by the Austrian geometer Karl Strubecker in the 1930's \cite{StrubeckerSOA1941,StrubeckerMZ1942,StrubeckerMZ1942b,StrubeckerMZ1944,StrubeckerMZ1949} (see also \cite{Sachs1990} and references therein), while that of $\mathbb{I}^3_{\mathrm{p}}$ began only recently \cite{AydinTJM2018,daSilva2018IsoGaussMap}. Besides its mathematical interest  \cite{AydinJG2015,AydinTJM2018,KaracanTJM2017,YoonJG2017,YoonSymmetry2016}, see also the recent contributions by this Author \cite{daSilvaArXivIso2017,daSilva2018IsoGaussMap}, isotropic geometry finds applications in economics \cite{AydinTJM2016,chenKJM2014}, image processing \cite{koenderink2002image}, and shape interrogation \cite{PottmannCAGD1994}. In addition, this theory may prove useful in understanding the geometry of surfaces with zero mean curvature in semi-Riemannian spaces \cite{satoMJOU2019,satoArXiv2018}.

Here we are mainly interested in the geometry of invariant surfaces, which are generated by the action of 1-parameter subgroups of simply and pseudo isotropic rigid motions. Consequently, our investigation generalizes the study of revolution and helicoidal surfaces in Euclidean and Lorentzian geometries to the context of spaces with a singular metric. The enumeration of 1-parameter subgroups of rigid motions in $\mathbb{I}^3$ has been already done by Strubecker \cite{StrubeckerSOA1941}, see also Chap. 2 of \cite{Sachs1990}. However, in $\mathbb{I}^3$ much attention has been paid only to  helicoidal and revolution surfaces, \cite{YoonJG2017} and \cite{AydinAKUJSE2016,KaracanTJM2017,YoonSymmetry2016}, respectively, while revolution surfaces in $\mathbb{I}_{\mathrm{p}}^3$ were studied in \cite{AydinTJM2018}. Here we revisit the 1-parameter subgroups of simply isotropic rigid motions by exploiting their linear representation in the group of invertible matrices $GL(4,\mathbb{R})$, which we believe offers the advantage of being simpler and easier to follow. We also study 1-parameter subgroups of pseudo-isotropic rigid motions and the invariant surfaces generated by them. After listing all invariant surfaces, which are divided into 7 basic types, we compute their mean and Gaussian curvatures and we also solve the problem of prescribed  curvatures for the so-called  invariant surfaces of i-type (see Definition \ref{def::InvSurfniAndiType}; for the ni-type we solve the prescribed Gaussian curvature problem for helicoidal surfaces only). These findings generalize the study of  helicoidal surfaces in $\mathbb{I}^3$ \cite{AydinAKUJSE2016,KaracanTJM2017} and  revolution surfaces in $\mathbb{I}_{\mathrm{p}}^3$
\cite{AydinTJM2018} with constant curvatures.

This work is divided as follows. In Sect. \ref{sec::Preliminaries} we present background material: definition of isotropic spaces in Subsect. \ref{subsec::SimplyAndPseudoIsoSpaces}; isotropic surfaces in Subsect. \ref{subsec::IsotropicSurf}; and the notion of invariant surfaces in Subsect. \ref{subsec::1ParGroupAndInvSurf}. In Sect. \ref{sec::1ParSubgroupIsotropicSymm} we describe 1-parameter subgroups of simply (Subsect. \ref{subsec::1ParSGSimplyIso}) and of pseudo (Subsect. \ref{subsec::1ParSGPseudoIso}) isotropic isometries, while in Sect. \ref{sec::SimplyPseudoIsoInvSurf} we describe their invariant surfaces. 
Finally, in Sect. \ref{sec::DGSimplyInvSurf} and \ref{sec::DGPseudoInvSurf} we study the geometry of simply and pseudo isotropic invariant surfaces, respectively, and solve the corresponding prescribed curvature problems. 

We mention that here we shall follow the Einstein convention of summing on repeated indexes, e.g., $A_i^k\mathbf{x}_k:=\sum_{k=1}^2A_i^k\mathbf{x}_k$. Also, most of the deductions in $\mathbb{I}_{\mathrm{p}}^3$ are omitted, the reader is then referred to the corresponding results in $\mathbb{I}^3$ to devise a proof.

\section{Preliminaries}
\label{sec::Preliminaries}

Simply and pseudo isotropic  spaces are examples of \textit{Cayley-Klein geometries} \cite{Giering1982}, i.e., the study of properties in projective space $\mathbb{P}^3$ invariant by the action of the projectivities that fix the so-called \textit{absolute figure}, which for our isotropic geometries are given by a plane at infinity, identified with $\omega:x_0=0$, and a degenerate quadric of index zero or one, identified with $Q:x_0^2+x_1^2+\delta\,x_2^2=0$: $\delta=1$ for the $\mathbb{I}^3$ absolute figure \cite{Sachs1990,StrubeckerSOA1941} and $\delta=-1$ for the $\mathbb{I}_{\mathrm{p}}^3$ one \cite{daSilvaArXivIso2017}. (Euclidean and Lorentzian geometries stand for the choice of $\omega:x_0=0$ and $Q:x_0^2+x_1^2+x_2^2+\delta\,x_3^2=0,\,\delta=\pm1$.)

\subsection{Simply and pseudo isotropic three-dimensional spaces}
\label{subsec::SimplyAndPseudoIsoSpaces}

In the Cayley-Klein framework, the \textit{simply isotropic} $\mathbb{I}^3$ and \textit{pseudo-isotropic} $\mathbb{I}_{\mathrm{p}}^3$ geometries are the study of those properties in $\mathbb{R}^3$ invariant by the action of the 6-parameter groups $\mathcal{B}_6$ \cite{Sachs1990,StrubeckerSOA1941} and $\mathcal{B}_6^{\mathrm{p}}$ \cite{daSilvaArXivIso2017} given by
\begin{equation}
\left\{
\begin{array}{cc}
\bar{x}  = & a + x\cos\phi-y\sin\phi\\ 
\bar{y}  = & b + x\sin\phi+y\cos\phi\\
\bar{z}  = & c + c_1x+c_2y+z\\
\end{array}
\right.\mbox{and}\left\{
\begin{array}{cc}
\bar{x}  = & a + x\cosh\phi+y\sinh\phi\\ 
\bar{y}  = & b + x\sinh\phi+y\cosh\phi\\
\bar{z}  = & c + c_1x+c_2y+z\\
\end{array}
\right.,\label{eq::SemiIsoGroupSB6}
\end{equation}
respectively, where $a,b,c,c_1,c_2,\phi\in\mathbb{R}$. 

On the $xy$ plane $\mathbb{I}^3$ and $\mathbb{I}_{\mathrm{p}}^2$ look just like the Euclidean $\mathbb{E}^2$ and Lorentzian $\mathbb{E}_1^2$ plane geometries. The projection of a vector $\mathbf{u}=(u^1,u^2,u^3)$ on the $xy$ plane  is called the \emph{top view} of $\mathbf{u}$ and it is denoted by $\widetilde{\mathbf{u}}=(u^1,u^2,0)$. Notice that the $z$-direction is preserved by the action of $\mathcal{B}_6$ and $\mathcal{B}_6^{\mathrm{p}}$. A line $P+t(0,0,a)$ is an \emph{isotropic line} and a plane containing an isotropic line is  an \emph{isotropic plane}.

One may respectively introduce a \emph{simply isotropic} and a \emph{pseudo-isotropic inner products} between two vectors $\mathbf{u}=(u^1,u^2,u^3)$ and $\mathbf{v}=(v^1,v^2,v^3)$ as
\begin{equation}
\langle\mathbf{u},\mathbf{v}\rangle_{z} = u^1v^1+u^2v^2
\mbox{ and }
\langle\mathbf{u},\mathbf{v}\rangle_{pz} = u^1v^1-u^2v^2.
\end{equation}
These inner products induce in $\mathbb{I}^3$ and $\mathbb{I}_{\mathrm{p}}^3$ a (semi) norm in a natural way: 
\begin{equation}
\Vert \mathbf{u}\Vert_z = \sqrt{\langle\mathbf{u},\mathbf{u}\rangle_z} = \Vert\widetilde{\mathbf{u}}\Vert
\mbox{ and }
\Vert \mathbf{u}\Vert_{pz} = \sqrt{\langle\mathbf{u},\mathbf{u}\rangle_{pz}} = \Vert\widetilde{\mathbf{u}}\Vert_1,
\end{equation}
respectively: $\Vert\cdot\Vert$ and $\Vert\cdot\Vert_1$ are the norms in $\mathbb{E}^3$ and $\mathbb{E}_1^3$ induced by
\begin{equation}\label{eq::EuclLorInnerProd}
\langle \mathbf{u},\mathbf{v}\rangle=u^1v^1+u^2v^2+u^3v^3
\mbox{ and }
\langle \mathbf{u},\mathbf{v}\rangle_1=u^1v^1-u^2v^2+u^3v^3.
\end{equation}
In addition, the corresponding vector products in $\mathbb{E}^3$ and $\mathbb{E}_1^3$ shall be denoted here by $\times$ and $\times_1$, respectively: notice that $\times_1$ should be computed as $$\mathbf{u}\times_1\mathbf{v}=(u^2v^3-u^3v^2,u^1v^3-u^3v^1,u^1v^2-u^2v^1).$$

The isotropic inner products are both degenerated: for any $\mathbf{u}=(0,0,a)$ we have $\mbox{span}\{\mathbf{u}\}^{\perp}=\mathbb{I}^3$ (or $\mathbb{I}_{\mathrm{p}}^3$). To exert some control on isotropic vectors, we  introduce the \textit{co-metric} $\langle\langle\mathbf{u},\mathbf{v}\rangle\rangle=u^3v^3$ coming from the \textit{codistance} 
$\mathrm{cd}(\mathbf{u},\mathbf{v})=\vert b^3-a^3\vert$, which is preserved by $\mathcal{B}_6$ and $\mathcal{B}_6^{\mathrm{p}}$ only when applied to isotropic vectors. In short, we think of $\mathbb{I}^3$ and $\mathbb{I}_{\mathrm{p}}^3$ as $\mathbb{R}^3$ equipped with an ``hierarchy" of metrics: when a metric fails to ``control" tangent vectors, we descend one level and start to use the next one.

\begin{rmk}
The hierarchy of metrics $(\langle\cdot,\cdot\rangle_z,\,\langle\langle\cdot,\cdot\rangle\rangle)$ in $\mathbb{I}^3$ and $\mathbb{I}_{\mathrm{p}}^3$ should not be confused with a single metric defined by parts, since a function like $G(\mathbf{u},\mathbf{v})=
\langle\mathbf{u},\mathbf{v}\rangle_z$, if $\tilde{\mathbf{u}}\not=0$ or $\tilde{\mathbf{u}}\not=0$, and $G(\mathbf{u},\mathbf{v})=
\langle\langle\mathbf{u},\mathbf{v}\rangle\rangle$, if $\tilde{\mathbf{u}}=\tilde{\mathbf{v}}=0$, is not bilinear and \emph{can not} be a metric. For example, we have $G((1,0,1)+(0,0,1),(0,0,1))=0$, but $G((1,0,1),(0,0,1))+G(0,0,1),(0,0,1))=1\not=0$.
\end{rmk}

\subsection{Simply and pseudo isotropic surfaces}
\label{subsec::IsotropicSurf}

When dealing with surfaces $M^2$ in isotropic geometry we must distinguish between two cases. We say that $M^2$ is an \textit{admissible surface} when the metric in $M^2$ induced by $\langle\cdot,\cdot\rangle_z$, or $\langle\cdot,\cdot\rangle_{pz}$, has rank 2. Otherwise, it is not admissible. If $M^2$ is parameterized by a $C^2$ map $\mathbf{x}(u^1,u^2)=(x^1(u^1,u^2),x^2(u^1,u^2),x^3(u^1,u^2))$, then it is admissible if, and only if, 
$X_{12}=x_1^1x_2^2-x_2^1x_1^2\not=0,$ 
where $x_k^i=\partial x^i/\partial u^k$ and $X_{12}$ comes from the notation
$$
X_{ij}=\det\left(
\begin{array}{cc}
x_1^i & x_1^j \\[3pt]
x_2^i & x_2^j \\
\end{array}
\right),\,X=\left(
\begin{array}{ccc}
x_1^1 & x_1^2 & x_1^3\\[3pt]
x_2^1 & x_2^2 & x_2^3\\
\end{array}
\right).\label{def::DefXandXij}
$$
As a consequence, every admissible $C^2$ surface $M^2$ can be locally parameterized as $\mathbf{x}(u^1,u^2)=(u^1,u^2,Z(u^1,u^2))$: we say that $M$ is in its \textit{normal form}. 

\emph{Simply isotropic} and \emph{pseudo-isotropic spheres} are connected and irreducible surfaces of degree 2 given by the 4-parameter family 
\begin{equation}
x^2+\sigma y^2+2c_1x+2c_2y+2c_3z+c_4=0,\,c_i\in\mathbb{R},\label{eq::GenericIsotropicSpheres}
\end{equation}
where $\sigma\in\{-1,+1\}$: $\sigma=+1$ in $\mathbb{I}^3$ \cite{Sachs1990} and $\sigma=-1$ in $\mathbb{I}_{\mathrm{p}}^3$ \cite{daSilvaArXivIso2017}.  Up to a rigid motion, we can express an isotropic sphere in one of the two normal forms below:
\newline
(1) \emph{Spheres of parabolic type}:
$$
\Sigma^2(p):z = -\frac{1}{2p}(x^2+y^2)+\frac{p}{2}\,\mbox{ in }\mathbb{I}^3,\,\Sigma_1^2(p):
z = -\frac{1}{2p}(x^2-y^2)+\frac{p}{2}\,\mbox{ in }\mathbb{I}_{\mathrm{p}}^3, \label{eq::DefParabolicSphere}
$$
where $p\not=0$ is the radius of the sphere and it is invariant under rigid motions (the translation by $p/2$ above, if we compare our normal form with that of Sachs \cite{Sachs1990} in $\mathbb{I}^3$, is to guarantee that $\Sigma^2(p)$ and $\Sigma_1^2(p)$ have their foci located at the origin of the coordinate system); and
\newline
(2) \emph{Spheres of cylindrical type}: 
 $
x^2+y^2=r^2\,\mbox{ in }\mathbb{I}^3\mbox{ and }
x^2-y^2=\pm\, r^2\,\mbox{ in }\mathbb{I}_{\mathrm{p}}^3,
$ where the radius $r>0$ is a constant and invariant under rigid motions.

Only spheres of parabolic type are admissible and we use them to introduce Gauss maps for surfaces $M^2$ in $\mathbb{I}^3$ and $\mathbb{I}_{\mathrm{p}}^3$ \cite{daSilva2018IsoGaussMap}: the normal with respect to $\langle\cdot,\cdot\rangle_z$, or  $\langle\cdot,\cdot\rangle_{pz}$, is necessarily the co-unit vector field $\mathcal{N}=(0,0,\pm1)$ and, needless to say, a constant Gauss map adds nothing interesting to the theory.

The \emph{first fundamental form} of $\mathbf{x}:\mathcal{U}\to M^2\subset \mathbb{I}^3$ (or $\mathbb{I}_{\mathrm{p}}^3$) is defined as usual:
\begin{equation}
\mathrm{I}=g_{ij}\rmd u^i\rmd u^j,\,g_{ij} = \langle\mathbf{x}_i,\mathbf{x}_j\rangle_z\mbox{ or }g_{ij} = \langle\mathbf{x}_i,\mathbf{x}_j\rangle_{pz}.
\end{equation}
When $M^2$ is parameterized in its normal form, $\mathrm{I}=g_{ij}\rmd u^i\rmd u^j$ takes a simpler form
\begin{equation}\label{eq::1stFFinNormalForm}
\mathrm{I} = (\rmd u^1)^2+(\rmd u^2)^2,\mbox{ in }\mathbb{I}^3,\mbox{ and }\mathrm{I} = (\rmd u^1)^2-(\rmd u^2)^2,\mbox{ in }\mathbb{I}_{\mathrm{p}}^3,
\end{equation}
which, in particular, shows that every admissible surface is flat, i.e., its intrinsic curvature vanishes identically, see Eq. (8.50) of \cite{Sachs1990} in $\mathbb{I}^3$ and Eq. (59) of \cite{daSilva2018IsoGaussMap} in $\mathbb{I}_{\mathrm{p}}^3$. In addition, every admissible surface in $\mathbb{I}_{\mathrm{p}}^3$ is timelike, i.e., $g_{ij}$ is non-degenerated with index 1 \cite{AydinTJM2018,daSilva2018IsoGaussMap}. Regular admissible curves may be either \emph{spacelike}, $\langle\alpha',\alpha'\rangle_{pz}>0$, or \emph{timelike}, $\langle\alpha',\alpha'\rangle_{pz}<0$, while \emph{lightlike} curves, $\langle\alpha',\alpha'\rangle_{pz}=0$, are non-admissible \cite{daSilvaArXivIso2017}.

The unit spheres of parabolic type $\Sigma^2$ and $\Sigma_1^2$ below will play a role in isotropic geometries similar to that of  $\mathbb{S}^2$ in Euclidean geometry $\mathbb{E}^3$ and of $\mathbb{S}_1^2$ in Lorentz-Minkowski geometry $\mathbb{E}^3_1$. Indeed, fixing the unit spheres
\begin{equation}\label{eq::UnitSphPar}
\Sigma^2:z = -\frac{1}{2}(x^2+y^2)+\frac{1}{2}\mbox{ and }
\Sigma_1^2:z= -\frac{1}{2}(x^2-y^2)+\frac{1}{2},
\end{equation} 
the \emph{Gauss maps} $\xi:M^2\subset\mathbb{I}^3\to\Sigma^2$ and $\xi^{\mathrm{p}}:M^2\subset\mathbb{I}_{\mathrm{p}}^3\to\Sigma_1^2$ are defined as
\begin{equation}
\xi(u^1,u^2) = \frac{X_{23}}{X_{12}}\mathbf{e}_1+\frac{X_{31}}{X_{12}}\mathbf{e}_2+\frac{1}{2}\left\{1-\left[\left(\frac{X_{23}}{X_{12}}\right)^2+\left(\frac{X_{31}}{X_{12}}\right)^2\right]\right\}\mathbf{e}_3,\label{def::IsoGaussMap}
\end{equation}
 and  
\begin{equation}
\xi^{\mathrm{p}}(u^1,u^2) = \frac{X_{23}}{X_{12}}\mathbf{e}_1+\frac{X_{13}}{X_{12}}\mathbf{e}_2+\frac{1}{2}\left\{1-\left[\left(\frac{X_{23}}{X_{12}}\right)^2-\left(\frac{X_{13}}{X_{12}}\right)^2\right]\right\}\mathbf{e}_3,\label{def::PseudoIsoGaussMap}
\end{equation}
where $\{\mathbf{e}_i\}_{i=1}^3$ is the canonical basis of the vector space $\mathbb{R}^3$.

In $\mathbb{E}^3$ and $\mathbb{E}_1^3$ the Gauss maps are $\xi_{eucl}=\mathbf{x}_1\times\mathbf{x}_2\,\Vert\mathbf{x}_1\times\mathbf{x}_2\Vert^{-1}$ and $\xi_{lor}=\mathbf{x}_1\times_1\mathbf{x}_2\,\Vert\mathbf{x}_1\times_1\mathbf{x}_2\Vert_1^{-1}$, respectively. Note that the top view of $\xi$ above coincides with that of the \emph{relative normal}
\begin{equation}
N_h=\frac{\mathbf{x}_1\times\mathbf{x}_2}{\Vert\widetilde{\mathbf{x}}_1\times\widetilde{\mathbf{x}}_2\Vert}=\frac{X_{23}}{X_{12}}\mathbf{e}_1+\frac{X_{31}}{X_{12}}\mathbf{e}_2+\mathbf{e}_3.\label{eq::DefIsoNh}
\end{equation}
(The $z$-coordinate of $\xi$ was then adjusted to give $\xi\circ\mathbf{x}\in\Sigma$.) In $\mathbb{I}_{\mathrm{p}}^3$, we have
\begin{equation}
N_h=\frac{\mathbf{x}_1\times_1\mathbf{x}_2}{\Vert\widetilde{\mathbf{x}}_1\times_1\widetilde{\mathbf{x}}_2\Vert_1}=\frac{X_{23}}{X_{12}}\mathbf{e}_1+\frac{X_{13}}{X_{12}}\mathbf{e}_2+\mathbf{e}_3\Rightarrow \widetilde{N}_h=\widetilde{\xi}^{\mathrm{p}}.\label{eq::DefPseudoNh}
\end{equation}

The \emph{isotropic shape operator} $L_q:T_qM^2\to T_qM^2$ is defined as
\begin{equation}
L_q(w_q)=-D_{w_q}\xi \mbox{ or }
L_q(w_q)=-D_{w_q}\,\xi^{\mathrm{p}},
\end{equation}
where $D$ denotes the usual (flat connection) directional derivative in $\mathbb{R}^3$. 
The shape operator $L_q$ maps $T_qM^2$ in $T_{\xi(q)}\Sigma^2$, but as in Euclidean space, $T_qM^2$ and $T_{\xi(q)}\Sigma^2$ are parallel and, therefore, they can be identified. The same reasoning applies to $\mathbb{I}_{\mathrm{p}}^3$ \cite{daSilva2018IsoGaussMap}. (In other words, $\xi$ is an equiaffine transversally vector field on $M^2$ \cite{Nomizu}.)

The \emph{second fundamental form} $\mathrm{II}$ in simply and pseudo-isotropic spaces is
\begin{equation}
\mathrm{II}(u_q,v_q) = \mathrm{I}(L_q(u_q),v_q)\Rightarrow \mathrm{II}=h_{ij}\rmd u^i\rmd u^j,\,h_{ij}=\mathrm{II}(\mathbf{x}_i,\mathbf{x}_j).
\end{equation}
The coefficients of $\mathrm{II}$ can be easily computed via $N_h$ in Eqs. (\ref{eq::DefIsoNh}) and (\ref{eq::DefPseudoNh}) \cite{daSilva2018IsoGaussMap}:
\begin{equation}
h_{ij} = \langle N_h,\mathbf{x}_{ij}\rangle\mbox{ in }\mathbb{I}^3\mbox{ and }h_{ij} = \langle N_h,\mathbf{x}_{ij}\rangle_1\mbox{ in }\mathbb{I}^3_{\mathrm{p}},
\end{equation}
where $\langle\cdot,\cdot\rangle$ and $\langle\cdot,\cdot\rangle_1$ are the inner products from $\mathbb{E}^3$ and $\mathbb{E}_1^3$ in Eq. (\ref{eq::EuclLorInnerProd}).

Finally, the \emph{isotropic Gaussian} and \emph{mean curvatures}  are defined as
\begin{equation}
K(q) = \det(L_q)\mbox{ and }H(q)=\frac{1}{2}\mathrm{tr}(L_q)\,\label{def::IsoGaussAndMeanCurv},
\end{equation}
respectively. In addition, if in local coordinates $L_q(\mathbf{x}_i)=-A_i^k\,\mathbf{x}_k$, then
\begin{equation}
h_{ij} = \mathrm{I}(L_q(\mathbf{x}_i),\mathbf{x}_j)=-A_i^k\,\mathrm{I}(\mathbf{x}_k,\mathbf{x}_j) = -A_i^k\,g_{kj}.\label{eq::CoefShapeOpe}
\end{equation}
From this relation it follows that $-A_i^k=g^{kj}h_{ji}$ and, therefore,
\begin{equation}
K = \frac{h_{11}h_{22}-h_{12}^2}{g_{11}g_{22}-g_{12}^2}\,\mbox{ and }\,H = \frac{1}{2}\frac{g_{11}h_{22}-2g_{12}h_{12}+g_{22}h_{11}}{g_{11}g_{22}-g_{12}^2}\,.
\end{equation}

The reader is referred to \cite{daSilva2018IsoGaussMap} for examples and more information about surfaces in isotropic geometries. (The geometry of curves can be found, e.g., in \cite{daSilvaArXivIso2017}.)

\subsection{One-parameter subgroups and invariant surfaces}
\label{subsec::1ParGroupAndInvSurf}

Let $(G,\circ)$ be a (Lie) group. A \textit{1-parameter subgroup} $H$ of $G$ is a subgroup of $G$ such that there exists a surjective continuous group homomorphism $\psi:(\mathbb{R},+)\to (H,\circ)$: $\psi_{r+s}=\psi_r\circ\psi_s, \,\psi_t=\psi(t).
$ Despite the non-uniqueness of $\psi$, it is common to identify $H$ with $\psi$, since $\psi(\mathbb{R})=H$. Sometimes it is useful to work with a linear representation of the  group of interest, i.e., see the elements of $G$ in a subgroup of invertible matrices of $GL(n,\mathbb{R})$. In this case, we have the following useful property

\begin{proposition}\label{prop::DetElemIn1ParSubG}
Let $H$ be a 1-parameter subgroup of $G\subseteq GL(n,\mathbb{R})$ and $h\in H$, then $\det h>0$.
\end{proposition}
\begin{proof}
If $h\in H$, then there exists a  map $\psi$ such that $\psi(1)=h$ (if $\psi(r)=h$, then  $\Psi(s)=\psi(rs)$ defines a new group homomorphism which sends $1$ to $h$). Since
$h=\psi(1)=\psi(\frac{1}{2}+\frac{1}{2})=\psi(\frac{1}{2})^2,
$ we have $\det h=[\det\psi(\frac{1}{2})]^2>0$.
\end{proof}

When $G$ happens to be a group of isometries of a certain geometry, its 1-parameter subgroups $\psi_t$ give rise to invariant surfaces:  a $C^2$ surface $M^2$ is said to be an \textit{invariant surface} when there exists a 1-parameter subgroup $\psi_t$ such that
\begin{equation}\label{eq::DefInvariantSurf}
\forall\,t\in\mathbb{R},\,M = \psi_t(M).
\end{equation}

Intuitively, we can approximate an invariant surface by successive applications, to a given curve $\alpha(s)$, of a certain kind of rigid motion (isometry): $$M\cong\{\dots,\psi_{t_0}(\alpha(s)),\psi_{t_0+\Delta t}(\alpha(s)),\dots,\psi_{t_0+n\Delta t}(\alpha(s)),\dots\}.$$ In the limit $\Delta t\to 0$, we generate $M$ by continuously moving the curve $\alpha$ under the action of  $\psi_t$. We call such $\alpha$ the \textit{generating curve} of $M$ and, for the geometries modeled in $\mathbb{R}^3$, it may be assumed to be planar. Usually, $\alpha$ is obtained by intersecting $M$ with the $xz$- or the $xy$-plane. Notice, if $\alpha(u)$ is the generating curve of  $M$, then one may parameterize $M$ by $\mathbf{x}(u,t)=\psi_t(\alpha(u))$. 

An important feature of invariant surfaces is that the values of geometric quantities, such as the Gaussian and mean curvatures, only depend on their values assumed along the generating curve.

\section{One-parameter subgroups of simply and pseudo isotropic isometries}
\label{sec::1ParSubgroupIsotropicSymm}

Let us introduce the following group of $4\times4$ invertible real matrices
\begin{equation}
\mathrm{G}\mathbb{I}(4)=\{M\in GL(4,\mathbb{R})\,:\,M=
\left(\begin{array}{cc}
A & \mathbf{a} \\
0 & 1 \\
\end{array}\right),\,\mathbf{a}\in\mathbb{I}^3\,,A\in \OIn{3}\},
\end{equation}
where $\OIn{3}$ denotes the set of simply isotropic orthogonal matrices:
\begin{equation}
A\in\OIn{3}\Leftrightarrow \forall\,u,v\in\mathbb{I}^3,\,\left\{
\begin{array}{c}
\langle Au,Av\rangle_z=\langle u,v\rangle_z\\
\textrm{cd}_z(Au,Av)=\textrm{cd}_z(u,v),\,\mbox{if }\Vert u\Vert_z=\Vert v\Vert_z=0\,
\end{array}
\right..
\end{equation}
Note that every element in $\OIn{3}$ can be written as
\begin{equation}\label{eq::ElementsOfOI3}
\left(\begin{array}{ccc}
\cos\phi & -\sigma\sin\phi & 0\\
\sin\phi & \sigma\cos\phi & 0\\
a & b & \tau\\
\end{array}\right),\,\phi,\,a,\,b\in\mathbb{R}\mbox{ and }\sigma,\tau\in\{-1,+1\}.
\end{equation}

Together with translations, $\OIn{3}$ gives us the group of simply isotropic isometries $\mathrm{ISO}(\mathbb{I}^3)$. Now we construct a linear representation $\Psi:\mathrm{ISO}(\mathbb{I}^3)\to \mathrm{G}\mathbb{I}(4)$ for the group $\mathrm{ISO}(\mathbb{I}^3)$. Thus, we can see isometries as  linear transformations by identifying $\mathbb{R}^3$ with a hyperplane in $\mathbb{R}^4$ via the inclusion map $(x,y,z)\in\mathbb{R}^3\mapsto(x,y,z,1)\in\mathbb{R}^4$.

\begin{proposition}
Let $\mathrm{ISO}(\mathbb{I}^3)$ be the group of simply isotropic isometries. Then,
\begin{enumerate}
\item $T\in \mathrm{ISO}(\mathbb{I}^3)\Leftrightarrow T(x)=Ax+\mathbf{a}$, where $A\in\OIn{3}$ and $\mathbf{a}\in\mathbb{I}^3$;
\item The map $\Psi$ below is a group isomorphism:
\begin{equation}
\begin{array}{cccc}
\Psi: & \mathrm{ISO}(\mathbb{I}^3) & \to & \mathrm{G}\mathbb{I}(4)\\
 & T & \mapsto & \left(\begin{array}{cc}
A & \mathbf{a} \\
0 & 1 \\
\end{array}\right)
\end{array}.
\end{equation}
\end{enumerate}
\label{Prop::StructureOfIsoIsometries}
\end{proposition}
\begin{proof}
(1) Given $T\in \mathrm{ISO}(\mathbb{I}^3)$, define $A(x)=T(x)-T(0)$. Using that $T$ preserves $\langle\cdot,\cdot\rangle_z$, and also $\mathrm{cd}_z(\cdot,\cdot)$ for isotropic vectors, one can deduce that $A\in\OIn{3}$. Indeed, expanding $L=\langle A(x+\lambda y)-A(x)-\lambda A(y),A(x+\lambda y)-A(x)-\lambda A(y)\rangle_z$, and using that $T$ preserves $\langle\cdot,\cdot\rangle_z$, we conclude that $L=0$. Applying the same reasoning to the codistance, we deduce that $A$ is a linear isometry, since $x=0\Leftrightarrow \langle x,x\rangle_z=0=\mathrm{cd}_z(x,0)$. Finally, writing $\mathbf{a}=T(0)$ we get the desired result. Conversely, since by definition any translation and any $A\in\OIn{3}$ give rise to isotropic isometries, every map $Ax+\mathbf{a}$ defines an isotropic rigid motion.

(2) $\Psi$ is clearly injective. Using item 1 we also see that $\Psi$ is surjective. Finally, if $T_i(x)=A_ix+\mathbf{a}_i$, then $(T_1\circ T_2)(x)=(A_1A_2)x+(A_1\mathbf{a}_2+\mathbf{a}_1)$ and 
\begin{eqnarray*}
\Psi(T_1)\Psi(T_2) &=& \left(
\begin{array}{cc}
A_1 & \mathbf{a}_1\\
0 & 1\\
\end{array}
\right)
\left(
\begin{array}{cc}
A_2 & \mathbf{a}_2\\
0 & 1\\
\end{array}
\right)\\
&=&\left(
\begin{array}{cc}
A_1A_2 & (A_1\mathbf{a}_2+\mathbf{a}_1)\\
0 & 1\\
\end{array}
\right)=
\Psi(T_1\circ T_2),
\end{eqnarray*}
 which shows that $\Psi$ is a group isomorphism.
\end{proof}

Notice that the same concepts above apply to rigid motions in  $\mathbb{I}_{\mathrm{p}}^3$. Indeed, define
\begin{equation}
\mathrm{G}\mathbb{I}_p(4)=\{M\in GL(4,\mathbb{R})\,:\,M=
\left(\begin{array}{cc}
A & \mathbf{a} \\
0 & 1 \\
\end{array}\right),\,a\in\mathbb{I}_{\mathrm{p}}^3\,,A\in \OIpn{3}\},
\end{equation}
where $\OIpn{3}$ denotes the set of pseudo-isotropic orthogonal matrices:
\begin{equation}
A\in\OIpn{3}\Leftrightarrow \,\left\{\begin{array}{c}
\langle Au,Av\rangle_{z,p}=\langle u,v\rangle_{z,p}\\
\textrm{cd}_{z,p}(Au,Av)=\textrm{cd}_{pz}(u,v),\,\mbox{if }\Vert u\Vert_{pz}=\Vert v\Vert_{pz}=0
\end{array}
\right..
\end{equation}
Therefore, every element in $\OIpn{3}$ can be written either as
\begin{equation}
\left(\begin{array}{ccc}
\sigma\cosh\phi & \sinh\phi & 0\\
\sinh\phi & \sigma\cosh\phi & 0\\
a & b & \tau\\
\end{array}\right)\,\mbox{or}\,\left(\begin{array}{ccc}
\sigma\cosh\phi & -\sinh\phi & 0\\
\sinh\phi & -\sigma\cosh\phi & 0\\
a & b & \tau\\
\end{array}\right),
\end{equation}
where $\phi,\,a,\,b\in\mathbb{R}$ and $\sigma,\tau\in\{-1,+1\}$. As in $\mathbb{I}^3$, we have in $\mathbb{I}_{\mathrm{p}}^3$ the 

\begin{proposition}
Let $\mathrm{ISO}(\mathbb{I}_{\mathrm{p}}^3)$ be the group of pseudo-isotropic isometries. Then,
\begin{enumerate}
\item $T\in \mathrm{ISO}(\mathbb{I}_{\mathrm{p}}^3)\Leftrightarrow T(x)=Ax+\mathbf{a}$, where $A\in\OIpn{3}$ and $\mathbf{a}\in\mathbb{I}_{\mathrm{p}}^3$;
\item The map $\Psi$ below is a group isomorphism:
\begin{equation}
\begin{array}{cccc}
\Psi: & \mathrm{ISO}(\mathbb{I}_{\mathrm{p}}^3) & \to & \mathrm{G}\mathbb{I}_{\mathrm{p}}(4)\\
 & T & \mapsto & \left(\begin{array}{cc}
A & \mathbf{a} \\
0 & 1 \\
\end{array}\right)
\end{array}.
\end{equation}
\end{enumerate}
\label{Prop::StructureOfPseudoIsoIsometries}
\end{proposition}

\subsection{One-parameter subgroups of simply isotropic isometries}
\label{subsec::1ParSGSimplyIso}

First, notice one may write an arbitrary element $h\in H\subseteq\mathrm{G}\mathbb{I}(4)$ as
\begin{equation}
h=\left(\arraycolsep=4pt\begin{array}{cccc}
\cos\phi & -\sin\phi & 0 & a\\
\sin\phi & \cos\phi & 0 & b \\
c_1 & c_2 & 1 & c\\
0 & 0 & 0 & 1\\
\end{array}\right).\label{eq::ElementsOfGI4}
\end{equation}
Indeed, since $\det h>0$ (Prop. \ref{prop::DetElemIn1ParSubG}), it follows that $\sigma\tau=+1$ in Eq. \eqref{eq::ElementsOfOI3}. In addition, given $\psi:\mathbb{R}\to \mathrm{G}\mathbb{I}(4)$, we may generate a 1-parameter subgroup of plane Euclidean isometries by composing $\psi_t$ with the map
$F$ that associates with any $T\in\mathrm{G}\mathbb{I}(4)$ the $3\times3$ matrix $F(T)=\left(\begin{array}{cc}
B & 0\\
0 & 1\\
\end{array}\right)$, where $B_{ij}=T_{ij}$, $i,j=1,2$. Then, $\sigma=+1$ and, consequently, $\tau=+1$. 

For a simply isotropic rigid motion what happens in the top view plane is  independent from what happens in the isotropic $z$-direction. Then, we may investigate the effect of a 1-parameter subgroup on the top view plane and on the isotropic direction separately. It will follow that the 1-parameter subgroups of simply isotropic isometries are distributed along 7 types, divided into two main categories \cite{Sachs1990,StrubeckerSOA1941}: (a) \emph{helicoidal motions}, which in the isotropic direction act either as a pure translation or as the identity map; and (b) \emph{limit motions} (\textit{Grenzbewegungen} \cite{StrubeckerSOA1941}), which in the top view plane act either as a pure translation or as the identity map.

By noticing that simply isotropic helicoidal motions are in one-to-one correspondence with Euclidean helicoidal motions with $\mathcal{O}z$ as the screw axis, the category of simply isotropic helicoidal motions is given by 
\begin{equation}
t\in\mathbb{R}\mapsto \psi_t=\left(\arraycolsep=4pt\begin{array}{cccc}
\cos(t\phi) & -\sin(t\phi) & 0 & 0\\
\sin(t\phi) & \cos(t\phi) & 0 & 0 \\
0 & 0 & 1 & c\,t\\
0 & 0 & 0 & 1\\
\end{array}\right)\in\mathrm{G}\mathbb{I}(4),
\end{equation}
which can be divided in two classes \cite{Sachs1990,StrubeckerSOA1941}: (I) \emph{Euclidean rotations} in the top view plane when $c=0$; and (II) \emph{helicoidal motions} around an isotropic axis.  

On the other hand, for the category of limit motions we have
\begin{thm}\label{thr::SimplyIsoG51parsubgroup}
The group $G_5$ of simply isotropic limit motions, i.e., motions that in the top view plane act either as a pure translation or as the identity map, leads to the 1-parameter subgroup
\begin{equation}
\psi_t=\left(\arraycolsep=4pt\begin{array}{cccc}
1 & 0 & 0 & a\,t\\[4pt]
0 & 1 & 0 & b\,t \\[4pt]
c_1t & c_2t & 1 & \displaystyle\left[c\,t+(ac_1+bc_2)\frac{t^2}{2}\right]\\[4pt]
0 & 0 & 0 & 1\\[4pt]
\end{array}\right)\in\mathrm{G}\mathbb{I}(4),\,c_i,c,a,b\in\mathbb{R}.\label{eq::Grenzbewegungen}
\end{equation}
\end{thm}
\begin{proof}
First, notice that the elements of $G_5$ are characterized by setting $\phi=0$ in Eq. (\ref{eq::ElementsOfGI4}). Now, if $h\in G_5$ and $\psi:\mathbb{R}\to H$ is a continuous surjective map such that $\psi(1)=h$, then one can prove by induction that
 $$\psi(n)=\psi(1)^n=\left(\arraycolsep=4pt\begin{array}{cccc}
1 & 0 & 0 & a\,n\\[4pt]
0 & 1 & 0 & b\,n \\[4pt]
c_1\,n & c_2\,n & 1 & \displaystyle\left[c\,n+\frac{n(n-1)}{2}\,(ac_1+bc_2)\right]\\[4pt]
0 & 0 & 0 & 1\\[4pt]
\end{array}\right).
$$
 In addition, since $h=\psi(1)=\psi(m/m)=\psi(1/m)^{m}$, we also have
 $$\psi\left(\frac{1}{m}\right)=\left(\arraycolsep=4pt\begin{array}{cccc}
1 & 0 & 0 & \displaystyle\frac{a}{m}\\[5pt]
0 & 1 & 0 & \displaystyle\frac{b}{m} \\[5pt]
\displaystyle\frac{c_1}{m} & \displaystyle\frac{c_2}{m} & 1 & \displaystyle\left[\frac{c}{m}+\frac{1}{2m}\left(\frac{1}{m}-1\right)(ac_1+bc_2)\right]\\[5pt]
0 & 0 & 0 & 1\\[4pt]
\end{array}\right).
$$
 Finally, for any rational $q=n/m$ we conclude that
$$
\psi\left(\frac{n}{m}\right)=\left(\arraycolsep=4pt\begin{array}{cccc}
1 & 0 & 0 & a\displaystyle\frac{n}{m}\\[5pt]
0 & 1 & 0 & b\displaystyle\frac{n}{m} \\[5pt]
c_1\displaystyle\frac{n}{m} & c_2\displaystyle\frac{n}{m} & 1 & \displaystyle\left[c\frac{n}{m}+\frac{n}{2m}\left(\frac{n}{m}-1\right)(ac_1+bc_2)\right]\\[5pt]
0 & 0 & 0 & 1\\[4pt]
\end{array}\right).
$$
Now, using the continuity of $\psi$ and, in addition, redefining $c$ to be $c+(ac_1+bc_2)/2$, we find Eq. (\ref{eq::Grenzbewegungen}). (The $(3,4)$-entry converges to $ct+(ac_1+bc_2)t(t-1)/2$, which under our redefinition of $c$ is equal to the $(3,4)$-entry of Eq. (\ref{eq::Grenzbewegungen}).)
\end{proof}

The 1-parameter group of limit motions may be divided in five types \cite{Sachs1990,StrubeckerSOA1941}: (III) \textit{parabolic rotations}, when $(ac_1+bc_2)\not=0$; (IV) \textit{warped translations}, when $c=(ac_1+bc_2)=0$, but $(a,b)\not=(0,0)$ and $(c_1,c_2)\not=(0,0)$; (V) \textit{isotropic shears}, when $(a,b)=(0,0)$ and $c=0$, but $(c_1,c_2)\not=(0,0)$; (VI) \textit{translations along a non-isotropic direction}, when $c,c_1,c_2=0$, but $(a,b)\not=(0,0)$; and (VII) \textit{translations along an isotropic direction}, when $(a,b),(c_1,c_2)=(0,0)$, but $c\not=0$. As we shall see in Sect. \ref{sec::SimplyPseudoIsoInvSurf}, only types (I), (II), (III), (IV), and (VI) lead to admissible surfaces.   

Finally, let us mention that the division into two classes according to the action of a 1-parameter subgroup in the top view and isotropic direction allowed us to classify the simpler contributions to each motion. However, if one does not impose any restriction on the action of $\psi_t$, the most general 1-parameter subgroup of isotropic isometries is given by 
\begin{equation}
\psi_t=\left(\begin{array}{cccc}
\cos(t\phi) & -\sin(t\phi) & 0 & aC_t-bS_t\\[4pt]
\sin(t\phi) & \cos(t\phi) & 0 & bC_t+aS_t \\[4pt]
c_1C_t+c_2S_t & c_2C_t-c_1S_t & 1 & ct+D_1\tilde{C}_t+D_2\tilde{S}_t\\[4pt]
0 & 0 & 0 & 1\\[4pt]
\end{array}\right),\label{eq::MostGeneral1parSimplyIsoMotion}
\end{equation}
where $D_1=(ac_1+bc_2)$, $D_2=(ac_2-bc_1)$, and we have defined the functions
$$
C_t(\phi) = \frac{1}{2}+\frac{\cos(t\phi-\phi)-\cos(t\phi)}{2[1-\cos(\phi)]},
\,S_t(\phi) = \frac{\sin(\phi)+\sin(t\phi-\phi)-\sin(t\phi)}{2[1-\cos(\phi)]},
$$
and
$$
\tilde{C}_t(\phi) = \frac{t}{2}-\frac{\cos(t\phi-\phi)-\cos(\phi)}{2[1-\cos(\phi)]},
\,
\tilde{S}_t(\phi) = \frac{t\sin(\phi)-\sin(t\phi-\phi)-\sin(\phi)}{2[1-\cos(\phi)]}.
$$
Notice that $\lim_{\phi\to0}C_t(\phi)=t$, $\lim_{\phi\to0}S_t(\phi)=0$, $\lim_{\phi\to0}\tilde{C}_t(\phi)=t(t-1)/2$, and $\lim_{\phi\to0}\tilde{S}_t(\phi)=0$, which then allow us to recover the known expression for a limit motion corresponding to $\phi=0$. On the other hand, it is easily seen that setting $a,b,c_1,c_2=0$ leads to helicoidal motions.

To prove Eq. \eqref{eq::MostGeneral1parSimplyIsoMotion}, let $\psi:\mathbb{R}\to H\subseteq \mathrm{G}\mathbb{I}(4)$ be a 1-parameter subgroup such that $\psi(1)=h$, with $h$ as in Eq. (\ref{eq::ElementsOfGI4}). Using that $\psi(n)=\psi(1)^n=h^n$, we find by induction
\begin{equation}
\psi(n)=\left(\arraycolsep=4pt\begin{array}{cccc}
\cos(n\phi) & -\sin(n\phi) & 0 & A_n(\phi)\\
\sin(n\phi) & \cos(n\phi) & 0 & B_n(\phi) \\
C_{1,n}(\phi) & C_{2,n}(\phi) & 1 & C_{0,n}(\phi)\\
0 & 0 & 0 & 1\\
\end{array}\right),
\end{equation}
where
$$
A_n(\phi)=a\sum_{k=0}^{n-1}\cos(k\phi)-b\sum_{k=0}^{n-1}\sin(k\phi)=aC_n(\phi)-bS_n(\phi),
$$
$$
B_n(\phi)=b\sum_{k=0}^{n-1}\cos(k\phi)+a\sum_{k=0}^{n-1}\sin(k\phi)=bC_n(\phi)+aS_n(\phi),
$$
\begin{eqnarray*}
C_{0,n}(\phi) & = & nc+D_1\sum_{k=0}^{n-2}(n-1-k)\cos(k\phi)+D_2\sum_{k=0}^{n-2}(n-1-k)\sin(k\phi)\nonumber\\
& = & nc+D_1\tilde{C}_n(\phi)+D_2\tilde{S}_n(\phi),
\end{eqnarray*}
$$
C_{1,n}(\phi)=c_1\sum_{k=0}^{n-1}\cos(k\phi)+c_2\sum_{k=0}^{n-1}\sin(k\phi)=c_1C_n(\phi)+c_2S_n(\phi),
$$
and
$$
C_{2,n}(\phi)=c_2\sum_{k=0}^{n-1}\cos(k\phi)-c_1\sum_{k=0}^{n-1}\sin(k\phi)=c_2C_n(\phi)-c_1S_n(\phi).
$$

Applying the same reasoning, we can find an expression for $\phi(1/m)=h^{1/m}$ (which allows us to find $\phi(r)$ for any rational $r$), extend it to the real numbers by continuity, and finally deduce Eq. (\ref{eq::MostGeneral1parSimplyIsoMotion}).

To deduce closed expressions for the sums $\sum \cos(k\phi)$, $\sum \sin(k\phi)$, $\sum (n-1-k)\cos(k\phi)$, and $\sum (n-1-k)\sin(k\phi)$ above, we may use complex numbers and then apply the known techniques of summing power series:
$$\left\{
\begin{array}{c}
\sum \cos(k\phi)+\mathrm{i}\sum \sin(k\phi)=\sum z^k\\[4pt]
\sum (n-1-k)\cos(k\phi)+\mathrm{i}\sum (n-1-j)\sin(k\phi)=\sum (n-1-k)z^k\\
\end{array}
\right.,
$$
where $z=\cos(\phi)+\mathrm{i}\sin(\phi)$.

\subsection{One-parameter subgroups of pseudo-isotropic isometries}
\label{subsec::1ParSGPseudoIso}

With a similar reasoning as in the  previous subsection, we may write any $h\in\mathrm{G}\mathbb{I}_{\mathrm{p}}(4)$ as 
\begin{equation}
h=\left(\arraycolsep=4pt\begin{array}{cccc}
\cosh\phi & \sinh\phi & 0 & a\\
\sinh\phi & \cosh\phi & 0 & b \\
c_1 & c_2 & 1 & c\\
0 & 0 & 0 & 1\\
\end{array}\right).\label{eq::ElementsOfGIp4}
\end{equation}

The 1-parameter subgroups of pseudo-isotropic isometries are also distributed along 7 types, divided into two categories: (a) \emph{causal helicoidal motions}, which in the isotropic direction act either as a pure translations or as the identity map; and (b) \emph{limit motions} (\emph{Grenzbewegungen}), which in the top view plane act either as a pure translation or as the identity map.

The category of causal helicoidal motions
\begin{equation}
t\in\mathbb{R}\mapsto \psi(t)=\left(\arraycolsep=4pt\begin{array}{cccc}
\cosh(t\phi) & \sinh(t\phi) & 0 & 0\\
\sinh(t\phi) & \cosh(t\phi) & 0 & 0 \\
0 & 0 & 1 & c\,t\\
0 & 0 & 0 & 1\\
\end{array}\right)\in\mathrm{G}\mathbb{I}_p(4)
\end{equation}
can be divided in two types: (I) \emph{Lorentzian rotations} in the top view plane, when $c=0$ (here the orbits can be either time- or space-like Lorentzian circles, i.e., hyperbolas); and (II) \emph{helicoidal motions} around an isotropic axis (here the orbits can be either time- or space-like Lorentzian helices).  

With a proof completely similar to that of Theorem \ref{thr::SimplyIsoG51parsubgroup}, the category of pseudo-isotropic limit motions can be described by
\begin{thm}\label{thr::PseudoIsoG51parsubgroup}
The group $G_5^{\mathrm{p}}$ of pseudo-isotropic limit motions, i.e., motions that in the top view plane act either as a pure translation or as the identity map, leads to the 1-parameter subgroup
\begin{equation}
 \psi_t=\left(\arraycolsep=4pt\begin{array}{cccc}
1 & 0 & 0 & a\,t\\[4pt]
0 & 1 & 0 & b\,t \\[4pt]
c_1t & c_2t & 1 & \displaystyle\left[c\,t+(ac_1+bc_2)\frac{t^2}{2}\right]\\[4pt]
0 & 0 & 0 & 1\\[4pt]
\end{array}\right)\in\mathrm{G}\mathbb{I}_p(4),\,c_i,c,a,b\in\mathbb{R}.\label{eq::PseudoGrenzbewegungen}
\end{equation}
\end{thm}

The 1-parameter subgroups of pseudo-isotropic limit motions are divided in five types: (III) \textit{pseudo-parabolic rotations}, when $(ac_1+bc_2)\not=0$; (IV) \textit{warped translations}, when $(ac_1+bc_2)=0$ and $c=0$, but $(a,b)\not=(0,0)$ and $(c_1,c_2)\not=(0,0)$; (V) \textit{pseudo-isotropic shears}, when $(a,b)=(0,0)$ and $c=0$, but $(c_1,c_2)\not=(0,0)$; (VI) \textit{translations along non-isotropic direction}, when $c=c_1=c_2=0$, but $(a,b)\not=(0,0)$; and (VII) \textit{translations along an isotropic direction}, when $(a,b),(c_1,c_2)=(0,0)$, but $c\not=0$. As we shall see in Sect. \ref{sec::SimplyPseudoIsoInvSurf}, only types (I), (II), (III), (IV), and (VI) lead to admissible surfaces. 

Finally, let us mention that without any restriction on $c,c_1,c_2,a,b$, and $\phi$, the most general 1-parameter subgroup of pseudo-isotropic isometries is
\begin{equation}\label{eq::MostGeneral1parPseudoIsoMotion}
\psi_t=\left(\begin{array}{cccc}
\cosh(t\phi) & \sinh(t\phi) & 0 & aCh_t-bSh_t\\[4pt]
\sinh(t\phi) & \cosh(t\phi) & 0 & bCh_t+aSh_t \\[4pt]
c_1Ch_t+c_2Sh_t & c_2Ch_t-c_1Sh_t & 1 & ct+D_1\tilde{Ch}_t+D_2\tilde{Sh}_t\\[4pt]
0 & 0 & 0 & 1\\[4pt]
\end{array}\right),
\end{equation}
where $D_1=(ac_1+bc_2)$, $D_2=(ac_2-bc_1)$, and we have defined the functions
$$
Ch_t(\phi) = \frac{1}{2}+\frac{\cosh(t\phi-\phi)-\cosh(t\phi)}{2[1-\cosh(\phi)]},
$$
$$
Sh_t(\phi) = \frac{\sinh(\phi)+\sinh(t\phi-\phi)-\sinh(t\phi)}{2[1-\cosh(\phi)]},
$$
$$
\tilde{Ch}_t(\phi) = \frac{t}{2}-\frac{\cosh(t\phi-\phi)-\cosh(\phi)}{2[1-\cosh(\phi)]},
$$
and
$$
\tilde{Sh}_t(\phi) = \frac{t\sinh(\phi)-\sinh(t\phi-\phi)-\sinh(\phi)}{2[1-\cosh(\phi)]}.
$$
Notice that $\lim_{\phi\to0}Ch_t(\phi)=t$, $\lim_{\phi\to0}Sh_t(\phi)=0$, $\lim_{\phi\to0}\tilde{Ch}_t(\phi)=t(t-1)/2$, and $\lim_{\phi\to0}\tilde{Sh}_t(\phi)=0$, which allow us to recover the expression for a limit motion corresponding to $\phi=0$.  

To deduce the expressions above, we can follow steps similar to those employed for $\mathrm{G}\mathbb{I}(4)$ by working with hyperbolic trigonometric functions. To find expressions for $\sum \cosh(k\phi)$, $\sum \sinh(k\phi)$, $\sum (n-1-k)\cosh(k\phi)$, and $\sum (n-1-k)\sinh(k\phi)$, we may use the ring of Lorentz numbers $\mathbb{L}$ \cite{daSilvaArXivIso2017} (also known as hyperbolic or double numbers \cite{Yaglom1979}) and then apply the known techniques of summing power series: $$
\left\{
\begin{array}{c}
\sum \cosh(k\phi)+\ell\sum \sinh(k\phi)=\sum w^k\\[4pt]
\sum (n-1-k)\cosh(k\phi)+\ell\sum (n-1-j)\sinh(k\phi)=\sum (n-1-k)w^k\\
\end{array}
\right.,$$
where $w=\cosh(\phi)+\ell\sinh(\phi)\in\mathbb{L}=\{a+\ell b:a,b\in\mathbb{R},\ell\not\in\mathbb{R},\ell^2=1\}.$

\section{Simply and pseudo isotropic invariant surfaces}
\label{sec::SimplyPseudoIsoInvSurf}

In this section, it is described the invariant surfaces obtained from  the 1-parameter subgroups of $\mathrm{ISO}(\mathbb{I}^3)$ (Subsect. \ref{subsec::SimplyInvSurf}) and of $\mathrm{ISO}(\mathbb{I}_{\mathrm{p}}^3)$ (Subsect. \ref{subsec::PseudoInvSurf}). We also identify those invariant surfaces that are admissible and show, as expected, that spheres of parabolic type are invariant surfaces under all types of revolutions. 

The generating curve of an invariant surface can be assumed to be a plane curve by intersecting the surface with a plane, in which case we should distinguish between isotropic and non-isotropic planes. Thus, we have the
\begin{dfn}\label{def::InvSurfniAndiType}
When the generating curve of an invariant surface, assumed to be at least $C^2$, lies on a non-isotropic plane (here we choose the $xy$-plane), we say that the corresponding invariant surface is of \textit{non-isotropic type}, or of \textit{ni-type}, for short. On the other hand, when the generating curve lies on an  isotropic plane (here we choose the $xz$-plane), we say that the corresponding invariant surface is of \textit{isotropic type}, or of \textit{i-type}, for short.
\end{dfn}

We shall see in Subsects. \ref{subsec::DGSimplyInvHelSurf} and \ref{subsec::DGPseudoInvHelSurf} that the mean curvature of helicoidal surfaces of ni-type depends on the value of $c$, while for the i-type there is no dependence on $c$. This shows that the distinction between invariant surfaces of ni- and i-types  in simply and pseudo isotropic geometries is meaningful. 

\subsection{Simply isotropic invariant surfaces}
\label{subsec::SimplyInvSurf}

We now use the 1-parameter subgroups of simply and pseudo isotropic isometries to investigate invariant surfaces, Eq. (\ref{eq::DefInvariantSurf}). For each 1-parameter subgroup we have two classes of surfaces, i.e., invariant surfaces of ni- and of i-types. 

We first describe helicoidal invariant surfaces in $\mathbb{I}^3$, i.e., when $a,b,c_i=0$:
\newline
(I) Euclidean revolution surfaces ($c=0$): 
\begin{itemize}
\item If $\alpha(u)=(x(u),y(u),0)$, we have the revolution surface of ni-type
\begin{equation}
Y_1(u,t) = (x(u)\cos(t\phi)-y(u)\sin(t\phi),x(u)\sin(t\phi)+y(u)\cos(t\phi),0), 
\end{equation}
that represents a portion of a non-isotropic plane;
\item If $\alpha(u)=(x(u),0,z(u))$, we have the revolution surface of i-type
\begin{equation}
Z_1(u,t) = (x(u)\cos(t\phi),x(u)\sin(t\phi),z(u)). 
\end{equation}
\end{itemize}
(II) Helicoidal surfaces:
\begin{itemize}
\item If $\alpha(u)=(x(u),y(u),0)$, we have the helicoidal surface of ni-type
\begin{equation}
Y_2(u,t) = (x(u)\cos(t\phi)-y(u)\sin(t\phi),x(u)\sin(t\phi)+y(u)\cos(t\phi),ct); 
\end{equation}
\item If $\alpha(u)=(x(u),0,z(u))$, we have the helicoidal surface of i-type
\begin{equation}
Z_2(u,t) = (x(u)\cos(t\phi),x(u)\sin(t\phi),z(u)+ct). 
\end{equation}
\end{itemize}

\begin{rmk}
When $\alpha$ is $\alpha(u)=R(\cos u,\sin u,0)$ in the ni-type or an isotropic line $\alpha(u)=(\pm R,0,z(u))$ in the i-type ($R$ constant), the resulting helicoidal surface is a sphere of cylindrical type, which is not an admissible surface.
\end{rmk}

\begin{proposition}
Assume that the generating curve is neither a circle centered at the origin  (in the ni-type) nor an isotropic line (in the i-type). Then, all simply isotropic Euclidean revolution and helicoidal surfaces are admissible. 
\end{proposition}
\begin{proof}
It is enough to investigate  helicoidal surfaces, type (II) above, since revolution surfaces are obtained by imposing $c=0$. In the ni-type we have 
$$\left\{
\begin{array}{lcl}
\partial_u Y_2 &=& (x'\cos(t\phi)-y'\sin(t\phi),x'\sin(t\phi)+y'\cos(t\phi),0) \\
\partial_t Y_2 &=& (-\phi x\sin(t\phi)-\phi y\cos(t\phi),\phi x\cos(t\phi)-\phi y \sin(t\phi),c)
\end{array}
\right.
$$ 
 and the first fundamental form is
\begin{equation}\label{eq::SimplyIso1stFFhelicoidalSurfniType}
\mathrm{I}=(x'\,^2+y'\,^2)\mathrm{d}u^2+2\phi(xy'-x'y)\mathrm{d}u\rmd t+\phi^2(x^2+y^2)\rmd t^2.
\end{equation}
The determinant of the induced metric is
\begin{equation}
\det g_{ij} = \phi^2(xx'+yy')^2=0\Leftrightarrow x^2+y^2=R^2.
\end{equation}
By hypothesis, $\alpha(u)$ is not a circle centered at the origin and, therefore, the corresponding helicoidal surface of ni-type is admissible.

On the other hand, in the i-type we have 
$$\left\{
\begin{array}{lcl}
\partial_u Z_2 &=& (x'\cos(t\phi),x'\sin(t\phi),z') \\
\partial_t Z_2 &=& (-\phi x\sin(t\phi),\phi x\cos(t\phi),c)
\end{array}
\right.
$$ 
 and the first fundamental form is
\begin{equation}\label{eq::SimplyIso1stFFhelicoidalSurfiType}
\mathrm{I}=x'\,^2\mathrm{d}u^2+\phi^2x^2\rmd t^2\Rightarrow\det g_{ij} = \phi^2(xx')^2=0\Leftrightarrow x^2=R^2.
\end{equation}
By hypothesis, $\alpha(u)$ is not an isotropic line and, therefore, the corresponding helicoidal surface of i-type is admissible.
\end{proof}
\begin{rmk}\label{rem::SimplyNormalFormHelSurf}
It is known that when written in its normal form, the first fundamental form of any simply  isotropic surface are given by Eq. (\ref{eq::1stFFinNormalForm}). Now notice that the first fundamental form of helicoidal surfaces of ni- and i-types can be written as $\mathrm{I}=\rmd U^2+\rmd T^2$ under the change of coordinates
$$
\left\{
\begin{array}{l}
U(u,t)=x(u)-t\phi\, y(u)+U_0\\
T(u,t)=y(u)+t\phi\, x(u)+T_0\\
\end{array}
\right.\mbox{ and }\left\{
\begin{array}{l}
U(u,t)=x(u)+U_0\\
T(u,t)=t\phi\, x(u)+T_0\\
\end{array}
\right.,
$$
respectively ($U_0,T_0$ constants). Indeed, completing  squares in  Eq. (\ref{eq::SimplyIso1stFFhelicoidalSurfniType}), gives
\begin{eqnarray}
\mathrm{I} & = & (x'^2\rmd u^2-2\phi x'y\rmd u\rmd t+\phi^2y^2\rmd t^2)+(y'^2\rmd u^2+2\phi xy'\rmd u\rmd t+\phi^2x^2\rmd t^2)\nonumber\\
& = & (x'\rmd u-\phi y\rmd t)^2+(y'\rmd u+\phi x\rmd t)^2.\nonumber
\end{eqnarray}
Now, introducing $U(u,t)$ and $T(u,t)$ such that 
$$\left\{
\begin{array}{c}
\rmd U = x'\rmd u-\phi y\rmd t\\
\rmd T = y'\rmd u+\phi x\rmd t
\end{array}
\right.\Rightarrow\left\vert\frac{\partial(U,T)}{\partial(u,t)}\right\vert=\left\vert
\begin{array}{lr}
x' & -\phi y\\
y' & \phi x\\
\end{array}
\right\vert=\phi(xx'+yy')\not=0,
$$
 we see that $(u,t)\mapsto (U,T)$ defines a smooth coordinate change, which can be easily integrated to give $U=x-t\phi y+U_0$ and $T=y+t\phi x+T_0$. On the other hand, for helicoidal surfaces of i-type it is seen that, under $(u,t)\mapsto (U=x+U_0,T=t\phi x+T_0)$, $\mathrm{I}$ in Eq. (\ref{eq::SimplyIso1stFFhelicoidalSurfiType}) is rewritten as $\mathrm{I}=\rmd U^2+\rmd T^2$.
\end{rmk}

Simply isotropic invariant surfaces resulting from limit motions are:
\newline
(III) Parabolic revolution surfaces:
\begin{itemize}
\item If $\alpha(u)=(x(u),y(u),0)$, we have the parabolic revolution surface of ni-type
\begin{equation}
Y_3(u,t) = (a\,t+x(u),b\,t+y(u),c\,t+\frac{ac_1+bc_2}{2}\,t^2+c_1t\,x(u)+c_2t\,y(u)); 
\end{equation}
\item If $\alpha(u)=(x(u),0,z(u))$, we have the parabolic revolution surface of i-type
\begin{equation}
Z_3(u,t) = (a\,t+x(u),b\,t,c\,t+\frac{ac_1+bc_2}{2}\,t^2+c_1t\,x(u)+z(u)). 
\end{equation}
\end{itemize}
(IV) Warped translation surfaces ($c=(ac_1+bc_2)=0$; $(a,b),(c_1,c_2)\not=(0,0)$):
\begin{itemize}
\item If $\alpha(u)=(x(u),y(u),0)$, we have the warped translation surface of ni-type
\begin{equation}
Y_4(u,t) = (a\,t+x(u),b\,t+y(u),c_1t\,x(u)+c_2t\,y(u)); 
\end{equation}
\item If $\alpha(u)=(x(u),0,z(u))$, we have the warped translation surface of i-type
\begin{equation}
Z_4(u,t) = (a\,t+x(u),b\,t,c_1t\,x(u)+z(u)). 
\end{equation}
\end{itemize}
(V) Isotropic shear surfaces ($a,b,c=0$; $(c_1,c_2)\not=(0,0)$):
\begin{itemize}
\item If $\alpha(u)=(x(u),y(u),0)$, we have the isotropic shear surface of ni-type
\begin{equation}
Y_5(u,t) = (x(u),y(u),c_1t\,x(u)+c_2t\,y(u)); 
\end{equation}
\item If $\alpha(u)=(x(u),0,z(u))$, we have the isotropic shear surface of i-type
\begin{equation}
Z_5(u,t) = (x(u),0,c_1t\,x(u)+z(u)).
\end{equation}
\end{itemize}
(VI) Translation surfaces along non-isotropic directions ($c,c_i=0$; $(a,b)\not=(0,0)$):
\begin{itemize}
\item If $\alpha(u)=(x(u),y(u),0)$, we have the translation surface of ni-type
\begin{equation}
Y_6(u,t) = (a\,t+x(u),b\,t+y(u),0), 
\end{equation}
that represents a portion of a non-isotropic plane;
\item If $\alpha(u)=(x(u),0,z(u))$, we have the translation surface of i-type
\begin{equation}
Z_6(u,t) = (a\,t+x(u),b\,t,z(u)), 
\end{equation}
that represents a cylindrical surface with non-isotropic rulings.
\end{itemize}
(VII) Translation surfaces along isotropic directions ($a,b,c_i=0$; $c\not=0$):
\begin{itemize}
\item If $\alpha(u)=(x(u),y(u),0)$, we have the translation surface of ni-type
\begin{equation}
Y_7(u,t) = (x(u),y(u),c\,t), 
\end{equation}
that represents a cylindrical surface with isotropic rulings;
\item If $\alpha(u)=(x(u),0,z(u))$, we have the translation surface of i-type
\begin{equation}
Z_7(u,t) = (x(u),0,z(u)+c\,t), 
\end{equation}
that represents a portion of an isotropic plane.
\end{itemize}

\begin{rmk}
Notice that isotropic shear surfaces and translation surfaces along isotropic directions, types (V) and (VII) above, are not admissible (here $a=0=b$). In addition, if $\alpha(u)=(au,bu,0)$, then the corresponding parabolic surface of ni-type is of the form $Y_3=(aU,bU,Z(U,V))$, where $U=t+u$ and $V=t-u$. Rewriting it as $Y_3=U\,(a,b,0)+Z\,\mathbf{e}_3$ clearly shows $Y_3$ is a portion of an isotropic plane. Similar conclusions apply for parabolic revolution surfaces of i-type when $\alpha(u)=(k,0,z(u))$ or when $b=0$. Indeed, in the former, $Z_3=(k,0,0)+t(a,b,0)+Z(t,u)\mathbf{e}_3$, while in the latter $Z_3=X(u,t)\mathbf{e}_1+Z(u,t)\mathbf{e}_3$.
\end{rmk}

\begin{proposition}\label{prop::SimplyIsoAdmissSurfForGrenzebewegung}
Assume that $a,b$ do not vanish simultaneously and also that the generating curve $\alpha$ is neither a part of a straight line $\{bX-aY=k,Z=0\}$ (in the ni-type) nor a part of an isotropic line or $b\not=0$ (in the i-type). Then, all simply isotropic parabolic revolution surfaces, warped translations surfaces, and translation surfaces along non-isotropic directions are admissible. 
\end{proposition}
\begin{proof}
It is enough to investigate the parabolic revolution surfaces, since the other surfaces are obtained by imposing some kind of restriction on the constants $a,b,c,c_1,$ and $c_2$. In the ni-type we have 
$$\left\{
\begin{array}{lcl}
\partial_u Y_3 &=& (x',y',c_1tx'+c_2ty') \\
\partial_t Y_3 &=& (a,b,c+(ac_1+bc_2)t+c_1x+c_2y)
\end{array}
\right.
$$
 and the first fundamental form is
\begin{equation}\label{eq::SimplyIso1stFFParaRevSurfniType}
\mathrm{I}=(x'\,^2+y'\,^2)\mathrm{d}u^2+2(ax'+by')\mathrm{d}u\rmd t+(a^2+b^2)\rmd t^2.
\end{equation}
The determinant of the induced metric is
\begin{equation}
\det g_{ij} = (bx'-ay')^2=0\Leftrightarrow bx-ay=k.
\end{equation}
By hypothesis, $\alpha(u)$ is not a straight line with neither slope $b/a$ nor $a/b$  and, therefore, the corresponding parabolic revolution surface of ni-type is admissible.

On the other hand, in the i-type we have 
$$\left\{
\begin{array}{lcl}
\partial_u Z_3 &=& (x',0,c_1tx'+z') \\
\partial_t Z_3 &=& (a,b,c+(ac_1+bc_2)t+c_1x)
\end{array}
\right.
$$
 and the first fundamental form is
\begin{equation}\label{eq::SimplyIso1stFFParaRevSurfiType}
\mathrm{I}=x'\,^2\mathrm{d}u^2+2ax'\rmd u\rmd t+(a^2+b^2)\rmd t^2\Rightarrow
\det g_{ij} = b^2x'\,^2=0\Leftrightarrow b=0\mbox{ or }x=k.
\end{equation}
By hypothesis, $\alpha(u)$ is neither an isotropic line nor $b\not=0$ and, therefore, the corresponding parabolic revolution surface of i-type is admissible.
\end{proof}

\begin{rmk}\label{rem::SimplyNormalFormParaSurf}
In analogy to Remark \ref{rem::SimplyNormalFormHelSurf}, the first fundamental form of parabolic  surfaces of ni- and i-types can be written as $\mathrm{I}=\rmd U^2+\rmd T^2$ under the map
$$
\left\{
\begin{array}{l}
U(u,t)=x(u)+at+U_0\\
T(u,t)=y(u)+bt+T_0\\
\end{array}
\right.\mbox{ and }\left\{
\begin{array}{l}
U(u,t)=x(u)+at+U_0\\
T(u,t)=bt+T_0\\
\end{array}
\right.,
$$
respectively ($U_0,T_0$ constants). Indeed, completing squares in Eq. (\ref{eq::SimplyIso1stFFParaRevSurfniType}), gives
\begin{eqnarray}
\mathrm{I} & = & (x'^2\rmd u^2+2ax'\rmd u\,\rmd t+a^2\rmd t^2)+(y'^2\rmd u^2+2by'\rmd u\,\rmd t+b^2\rmd t^2)\nonumber\\
& = & (x'\rmd u+a\,\rmd t)^2+(y'\rmd u+b\,\rmd t)^2.\nonumber
\end{eqnarray}
Now, introducing $U(u,t)$ and $T(u,t)$ such that 
$$\left\{
\begin{array}{c}
\rmd U = x'\rmd u+a\,\rmd t\\
\rmd T = y'\rmd u+b\,\rmd t
\end{array}
\right.\Rightarrow\left\vert\frac{\partial(U,T)}{\partial(u,t)}\right\vert=\left\vert
\arraycolsep=3pt\begin{array}{lr}
x' & a\\
y' & b\\
\end{array}
\right\vert=(bx'-ay')\not=0,
$$
 we see that $(u,t)\mapsto (U,T)$ defines a smooth coordinate change, which can be easily integrated to give $U=x+at+U_0$ and $T=y+bt+T_0$. On the other hand, proceeding similarly for parabolic surfaces of i-type, we  see that, under $(u,t)\mapsto (U=x+at+U_0,T=bt+T_0)$, $\mathrm{I}$ in Eq. (\ref{eq::SimplyIso1stFFParaRevSurfiType}) is rewritten as $\mathrm{I}=\rmd U^2+\rmd T^2$.
\end{rmk}

Concerning the most general 1-parameter subgroup of simply isotropic motion in Eq. (\ref{eq::MostGeneral1parSimplyIsoMotion}), we have the following invariant surfaces of ni- and i-types
\begin{eqnarray}
Y_8(u,t) & = & [x\,\cos(t\phi)-y\,\sin(t\phi)+aC_t-bS_t]\mathbf{e}_1+\nonumber\\
& + & [x\,\sin(t\phi)+y\,\cos(t\phi)+bC_t+aS_t]\mathbf{e}_2+ [c\,t +(ac_1+bc_2)\tilde{C}_t+\nonumber\\
& + & (ac_2-bc_1)\tilde{S}_t+(c_1C_t+c_2S_t)x+(c_2C_t-c_1S_t)y]\mathbf{e}_3, 
\end{eqnarray}
if the generating curve is $\alpha(u)=(x(u),y(u),0)$  and
\begin{eqnarray}
Z_8(u,t) & = & [x(u)\,\cos(t\phi)+aC_t-bS_t]\mathbf{e}_1+\nonumber\\
&+&[x(u)\,\sin(t\phi)+bC_t+aS_t]\mathbf{e}_2+[c\,t+(ac_1+bc_2)\tilde{C}_t+\nonumber\\
&+&(ac_2-bc_1)\tilde{S}_t+(c_1C_t+c_2S_t)x(u)+z(u)]\mathbf{e}_3, 
\end{eqnarray}
if the generating curve is $\alpha(u)=(x(u),0,z(u))$, respectively.

Given a point $q\in \mathbb{I}^3$, the orbits of the 1-parameter subgroups $\psi_t$ of type (I), (II), and (III), i.e., the curves $t\mapsto\psi_t(q)$, are Euclidean  circles, Euclidean  helices with an isotropic screw axis, and parabolas (isotropic circles), respectively. For the other types of 1-parameter subgroups, i.e., when $\phi=0$ and $(ac_1+bc_2)=0$, the orbits of the corresponding 1-parameter subgroups of limit motions are straight lines and, therefore, 
\begin{proposition}
Invariant surfaces in $\mathbb{I}^3$ of type $\mathrm{(IV)}-\mathrm{(VII)}$ are ruled. \label{Prop::SimpIsoInvarRuledSurf}
\end{proposition}

\begin{example}[Spheres as invariant surfaces]\label{exe::SphrAsInvSurf}
Let $p\not=0$ be a constant and $\alpha(u)=(pu,0,\frac{pu^2}{2})$ be an isotropic circle (the induced geometry in the $xz$-plane is $\mathbb{I}^2$, whose circles of parabolic type are precisely parabolas with  vertex and focus lying on an isotropic line \cite{Sachs1987}). Applying an elliptic revolution to $\alpha$ gives
$$Z_2(u,t)=(pu\cos(t\phi),pu\sin(t\phi),\frac{pu^2}{2}),$$
which parameterizes the sphere of parabolic type $\Sigma^2(p):z=\frac{1}{2p}(x^2+y^2)$. On the other hand, applying a parabolic revolution to $\alpha(u)$ with $c=0$, $a=pc_1$, and $b=pc_2$, gives
$$Z_3(u,t)=(at+pu,bt,\frac{(ac_1+bc_2)}{2}t^2+c_1tpu+\frac{pu^2}{2}),$$
which also parameterizes the sphere of parabolic type $\Sigma^2(p)$.
\end{example}

Up to a linear change of coordinates, in Euclidean space there exists only one type of rotation. In $\mathbb{I}^3$, however, we must distinguish between Euclidean and parabolic rotations based on their effect on the top view plane. The example above shows us that spheres of parabolic type are invariant with respect to all types of rotation we can have in simply isotropic geometry.

\subsection{Pseudo-isotropic invariant surfaces}
\label{subsec::PseudoInvSurf}

As in $\mathbb{I}^3$, for each kind of 1-parameter subgroup of $\mathrm{ISO}(\mathbb{I}^3_{\mathrm{p}})$ we distinguish the corresponding surface based on the type of plane containing the generating curve (Definition \ref{def::InvSurfniAndiType}). In addition, for the i-type we also distinguish between the  $xz$- and $yz$-planes, since they lead to curves of distinct causal characters: spacelike in the former and timelike in the latter. 

We first describe helicoidal invariant surfaces in $\mathbb{I}_{\mathrm{p}}^3$, i.e., when $a,b,c_i=0$:
\newline
(I) Lorentzian revolution surfaces ($c=0$): 
\begin{itemize}
\item If $\alpha(u)=(x(u),y(u),0)$, we have the revolution surface of ni-type
\begin{equation}
Yh_1(u,t) = (x(u)\cosh(t\phi)+y(u)\sinh(t\phi),x(u)\sinh(t\phi)+y(u)\cosh(t\phi),0), 
\end{equation}
that represents a portion of a non-isotropic plane;
\item If $\alpha(u)=(x(u),0,z(u))$ ($\alpha$ spacelike), we have the surface of i-type
\begin{equation}
Zh_1(u,t) = (x(u)\cosh(t\phi),x(u)\sinh(t\phi),z(u)); 
\end{equation}
\item If $\alpha(u)=(0,y(u),z(u))$ ($\alpha$ timelike), we have the surface of i-type
\begin{equation}
Wh_1(u,t) = (y(u)\sinh(t\phi),y(u)\cosh(t\phi),z(u)). 
\end{equation}
\end{itemize}
(II) Helicoidal surfaces:
\begin{itemize}
\item If $\alpha(u)=(x(u),y(u),0)$, we have the helicoidal surface of ni-type
\begin{equation}
Yh_2(u,t) = (x(u)\cosh(t\phi)+y(u)\sinh(t\phi),x(u)\sinh(t\phi)+y(u)\cosh(t\phi),ct); 
\end{equation}
\item If $\alpha(u)=(x(u),0,z(u))$ ($\alpha$ spacelike), we have the surface of i-type
\begin{equation}
Zh_2(u,t) = (x(u)\cosh(t\phi),x(u)\sinh(t\phi),z(u)+ct); 
\end{equation}
\item If $\alpha(u)=(0,y(u),z(u))$ ($\alpha$ timelike), we have the  surface of i-type
\begin{equation}
Wh_2(u,t) = (y(u)\sinh(t\phi),y(u)\cosh(t\phi),z(u)+ct). 
\end{equation}
\end{itemize}

\begin{rmk}
When $\alpha$ is $\alpha(u)=R(\cosh u,\sinh u,0)$ or $\alpha(u)=R(\sinh u,\cosh u,0)$ in the ni-type or  $\alpha(u)=(\pm R,0,z(u))$ in the i-type ($R$ constant), the resulting helicoidal surface is a sphere of cylindrical type, which is not admissible.
\end{rmk}

\begin{proposition}
Assume the generating curve $\alpha$ is neither a hyperbola centered at the origin (in the ni-type) nor an isotropic line (in the i-type). Then, all pseudo-isotropic Lorentzian revolution and helicoidal surfaces are admissible. 
\end{proposition}
\begin{proof}
It is enough to investigate helicoidal surfaces, since revolution surfaces are obtained by imposing $c=0$. In the ni-type 
 the first fundamental form is
\begin{equation}\label{eq::PseudoIso1stFFhelicoidalSurfniType}
\mathrm{I}=(x'\,^2-y'\,^2)\mathrm{d}u^2+2\phi(x'y-xy')\mathrm{d}u\rmd t-\phi^2(x^2-y^2)\rmd t^2.
\end{equation}
The determinant of the induced metric is
\begin{equation}
\det g_{ij} = -\phi^2(xx'-yy')^2=0\Leftrightarrow x^2-y^2=\pm R^2.
\end{equation}
By hypothesis, $\alpha(u)$ is not a hyperbola centered at the origin and, therefore, the corresponding helicoidal surface of ni-type is admissible.

On the other hand, in the i-type 
 the first fundamental form is
\begin{equation}\label{eq::ZhPseudoIso1stFFhelicoidalSurfiType}
\mathrm{I}=x'\,^2\mathrm{d}u^2-\phi^2x^2\rmd t^2\Rightarrow
\det g_{ij} = -\phi^2(xx')^2=0\Leftrightarrow x^2=R^2.
\end{equation}
By hypothesis, $\alpha(u)$ is not an isotropic line and, therefore, the corresponding helicoidal surface of i-type is admissible. Analogously, for $Wh_2$
\begin{equation}\label{eq::WhPseudoIso1stFFhelicoidalSurfiType}
\mathrm{I}=-y'\,^2\mathrm{d}u^2+\phi^2y^2\rmd t^2\Rightarrow \det g_{ij}=-\phi^2y^2y'\,^2=0\Leftrightarrow y^2= R^2.
\end{equation}
\end{proof}
\begin{rmk}
In analogy to  Remark \ref{rem::SimplyNormalFormHelSurf}, the first fundamental form of helicoidal surfaces of ni- and i-types can be respectively written as $\mathrm{I}=\rmd U^2-\rmd T^2$ and $\mathrm{I}=\rmd U\,\rmd T$ under the change of coordinates
$$
\left\{
\begin{array}{l}
U(u,t)=x(u)+t\phi\, y(u)+U_0\\
T(u,t)=y(u)+t\phi\, x(u)+T_0\\
\end{array}
\right.\mbox{ and }\left\{
\begin{array}{l}
U(u,t)=x(u)+t\phi\, x(u)+U_0\\
T(u,t)=x(u)-t\phi\, x(u)+T_0\\
\end{array}
\right.,
$$
respectively ($U_0,T_0$ constants). Indeed, we just need to proceed by completing squares in Eq. (\ref{eq::PseudoIso1stFFhelicoidalSurfniType}) for surfaces of ni-type and in 
 Eq. (\ref{eq::ZhPseudoIso1stFFhelicoidalSurfiType}) for the i-type. (Notice $\rmd U\,\rmd T$ is mapped to the normal form $\rmd U^2-\rmd T^2$ via $U\mapsto U+T$ and $T\mapsto U-T$.)
\end{rmk}

Pseudo-isotropic invariant surfaces resulting from limit motions are:
\newline
(III) Pseudo-parabolic revolution surfaces:
\begin{itemize}
\item If $\alpha(u)=(x(u),y(u),0)$, we have the revolution surface of ni-type
\begin{equation}
Yh_3(u,t) = (at+x(u),b\,t+y(u),ct+\frac{ac_1+bc_2}{2}\,t^2+c_1t\,x(u)+c_2t\,y(u)); 
\end{equation}
\item If $\alpha(u)=(x(u),0,z(u))$ ($\alpha$ spacelike), we have the surface of i-type
\begin{equation}
Zh_3(u,t) = (a\,t+x(u),b\,t,c\,t+\frac{ac_1+bc_2}{2}\,t^2+c_1t\,x(u)+z(u)); 
\end{equation}
\item If $\alpha(u)=(0,y(u),z(u))$ ($\alpha$ timelike), we have the surface of i-type
\begin{equation}
Wh_3(u,t) = (a\,t,b\,t+y(u),c\,t+\frac{ac_1+bc_2}{2}\,t^2+c_2t\,y(u)+z(u)). 
\end{equation}
\end{itemize}
(IV) Warped translation surfaces ($c,ac_1+bc_2=0$; $(a,b),(c_1,c_2)\not=(0,0)$):
\begin{itemize}
\item If $\alpha(u)=(x(u),y(u),0)$, we have the warped surface of ni-type
\begin{equation}
Yh_4(u,t) = (a\,t+x(u),b\,t+y(u),c_1t\,x(u)+c_2t\,y(u)); 
\end{equation}
\item If $\alpha(u)=(x(u),0,z(u))$ ($\alpha$ spacelike), we have the warped surface of i-type
\begin{equation}
Zh_4(u,t) = (a\,t+x(u),b\,t,c_1\,t\,x(u)+z(u)); 
\end{equation}
\item If $\alpha(u)=(0,y(u),z(u))$ ($\alpha$ timelike), we have the warped surface of i-type
\begin{equation}
Wh_4(u,t) = (a\,t,b\,t+y(u),c_2\,t\,y(u)+z(u)), 
\end{equation}
\end{itemize}
(V) Pseudo-isotropic shear surfaces ($c,a,b=0$; $(c_1,c_2)\not=(0,0)$):
\begin{itemize}
\item If $\alpha(u)=(x(u),y(u),0)$, we have the shear surface of ni-type
\begin{equation}
Yh_5(u,t) = (x(u),y(u),c_1t\,x(u)+c_2t\,y(u)); 
\end{equation}
\item If $\alpha(u)=(x(u),0,z(u))$ ($\alpha$ spacelike), we have the shear surface of i-type
\begin{equation}
Zh_5(u,t) = (x(u),0,c_1t\,x(u)+z(u));
\end{equation}
\item If $\alpha(u)=(0,y(u),z(u))$ ($\alpha$ timelike), we have the shear surface of i-type
\begin{equation}
Wh_5(u,t) = (0,y(u),c_2t\,y(u)+z(u)).
\end{equation}
\end{itemize}
(VI) Translation surfaces along non-isotropic directions ($c,c_i=0$; $(a,b)\not=(0,0)$):
\begin{itemize}
\item If $\alpha(u)=(x(u),y(u),0)$, we have the translation surface of ni-type
\begin{equation}
Yh_6(u,t) = (a\,t+x(u),b\,t+y(u),0), 
\end{equation}
that represents a portion of a non-isotropic plane;
\item If $\alpha(u)=(x(u),0,z(u))$ ($\alpha$ spacelike), we have the surface of i-type
\begin{equation}
Zh_6(u,t) = (a\,t+x(u),b\,t,z(u)); 
\end{equation}
\item If $\alpha(u)=(0,y(u),z(u))$ ($\alpha$ timelike), we have the surface of i-type
\begin{equation}
Wh_6(u,t) = (a\,t,b\,t+y(u),z(u)). 
\end{equation}
Both $Zh_6$ and $Wh_6$ represent a cylinder with non-isotropic rulings.
\end{itemize}
(VII) Translation surfaces along an isotropic direction ($a,b,c_i=0$; $c\not=0$):
\begin{itemize}
\item If $\alpha(u)=(x(u),y(u),0)$, we have the translation surface of ni-type
\begin{equation}
Yh_7(u,t) = (x(u),y(u),c\,t), 
\end{equation}
that represents a cylindrical surface with isotropic rulings;
\item If $\alpha(u)=(x(u),0,z(u))$ ($\alpha$ spacelike), we have the surface of i-type
\begin{equation}
Zh_7(u,t) = (x(u),0,z(u)+c\,t), 
\end{equation}
that represents a portion of an isotropic plane;
\item If $\alpha(u)=(0,y(u),z(u))$ ($\alpha$ timelike), we have the translation surface of i-type
\begin{equation}
Wh_7(u,t) = (0,y(u),z(u)+c\,t), 
\end{equation}
that also represents a portion of an isotropic plane.
\end{itemize}
\begin{rmk}
Notice that isotropic shear surfaces and translation surfaces along isotropic directions, types (V) and (VII) above, are not admissible (here $a=0=b$). In addition, if $\alpha(u)=(au,bu,0)$, then the corresponding  parabolic surface of ni-type is of the form $Yh_3=(aU,bU,Z(U,V))$, where $U=t+u$ and $V=t-u$. Rewriting it as $Yh_3=U\,(a,b,0)+Z\,\mathbf{e}_3$ clearly shows $Y_3$ is a portion of an isotropic plane. Similar conclusions apply for parabolic revolution surfaces of i-type when $\alpha(u)$ is an isotropic line or when $b=0$.
\end{rmk}

\begin{proposition}\label{prop::PseudoIsoAdmissSurfForGrenzebewegung}
Assume that $a,b$ do not vanish simultaneously and also that the generating curve $\alpha$ is neither a part of a straight line $\{bX-aY=k,Z=0\}$ (in the ni-type) nor a part of an isotropic line or $b\not=0$ (in the i-type). Then, all pseudo-isotropic parabolic revolution surfaces, warped translations surfaces, and translation surfaces along non-isotropic directions are admissible. 
\end{proposition}
\begin{proof}
It is enough to investigate the parabolic revolution surfaces, since the other surfaces are obtained by imposing some kind of restriction on the constants $a,b,c,c_1,$ and $c_2$. In the ni-type 
 the first fundamental form is
\begin{equation}\label{eq::PseudoIso1stFFParaRevSurfniType}
\mathrm{I}=(x'\,^2-y'\,^2)\mathrm{d}u^2+2(ax'-by')\mathrm{d}u\rmd t+(a^2-b^2)\rmd t^2.
\end{equation}
The determinant of the induced metric is
\begin{equation}
\det g_{ij} = -(bx'-ay')^2=0\Leftrightarrow bx-ay=k.
\end{equation}
By hypothesis, $\alpha(u)$ is not a straight line with neither slope $b/a$ nor $a/b$  and, thus, the corresponding parabolic revolution surface of ni-type is admissible.

On the other hand, in the i-type 
  the first fundamental form is
\begin{equation}\label{eq::ZhPseudoIso1stFFParaRevSurfiType}
\mathrm{I}=x'^2\mathrm{d}u^2+2ax'\rmd u\rmd t+(a^2-b^2)\rmd t^2
\Rightarrow\det g_{ij} = -b^2x'^2=0\Leftrightarrow b=0\,\mbox{or}\,x=k.
\end{equation}
By hypothesis, $\alpha(u)$ is neither an isotropic line nor $b\not=0$ and, therefore, the corresponding parabolic revolution surface of i-type is admissible.

The computations for $Wh_3$ are similar:
\begin{equation}\label{eq::WhPseudoIso1stFFParaRevSurfiType}
\mathrm{I}=-y'\,^2\mathrm{d}u^2-2by'\rmd u\rmd t+(a^2-b^2)\rmd t^2\Rightarrow \det g_{ij}=-a^2y'\,^2=0\Leftrightarrow x^2=R^2.
\end{equation}
\end{proof}
\begin{rmk}
In analogy to Remark \ref{rem::SimplyNormalFormParaSurf}, the first fundamental form of parabolic revolution surfaces of ni- and i-types can be written as $\mathrm{I}=\rmd U^2-\rmd T^2$ under the change of coordinates
$$
\left\{
\begin{array}{l}
U(u,t)=x(u)+at+U_0\\
T(u,t)=y(u)+bt+T_0\\
\end{array}
\right.\mbox{ and }\left\{
\begin{array}{l}
U(u,t)=x(u)+at+U_0\\
T(u,t)=bt\phi\, x(u)+T_0\\
\end{array}
\right.,
$$
respectively ($U_0,T_0$ constants). Indeed, we just proceed by completing squares  in Eq. (\ref{eq::PseudoIso1stFFParaRevSurfniType}) for surfaces of ni-type and  in 
 Eq. (\ref{eq::ZhPseudoIso1stFFParaRevSurfiType}) for surfaces of i-type.
\end{rmk}

Concerning the most general 1-parameter subgroup of pseudo-isotropic motion in Eq. (\ref{eq::MostGeneral1parPseudoIsoMotion}), we have the following invariant surfaces of ni- and i-types
\begin{equation}
Yh_8(u,t) = [x\,\cosh(t\phi)-y\,\sinh(t\phi)+aCh_t-bSh_t]\mathbf{e}_1+
\end{equation}
$$+ [x\sinh(t\phi)+y\cosh(t\phi)+bCh_t+aSh_t]\mathbf{e}_2+[ct+(ac_1+bc_2)\tilde{Ch}_t+$$
$$+(ac_2-bc_1)\tilde{Sh}_t+(c_1Ch_t+c_2Sh_t)x+(c_2Ch_t-c_1Sh_t)y]\mathbf{e}_3,$$
if the generating curve is $\alpha(u)=(x(u),y(u),0)$, and
\begin{equation}
Zh_8(u,t) = [x(u)\,\cosh(t\phi)+aCh_t-bSh_t]\mathbf{e}_1+
\end{equation}
$$+[x(u)\,\sinh(t\phi)+bCh_t+aSh_t]\mathbf{e}_2+[c\,t+(ac_1+bc_2)\tilde{Ch}_t+$$
$$+ (ac_2-bc_1)\tilde{Sh}_t+(c_1Ch_t+c_2Sh_t)x(u)+z(u)]\mathbf{e}_3, 
$$
if the generating curve is $\alpha(u)=(x(u),0,z(u))$.

Notice that given a point $q\in \mathbb{I}_{\mathrm{p}}^3$, the orbits of the 1-parameter subgroups $\psi_t$ of type (I), (II), and (III), i.e., the curves $t\mapsto\psi_t(q)$, are Lorentzian circles, Lorentzian  helices with an isotropic screw axis, and parabolas (isotropic circles), respectively. For the other types of 1-parameter subgroups, i.e., when $\phi=0$ and $(ac_1+bc_2)=0$, the orbits of the corresponding 1-parameter subgroups of limit motions are straight lines and, therefore, 
\begin{proposition}
Invariant surfaces in $\mathbb{I}_{\mathrm{p}}^3$ of type $\mathrm{(IV)-(VII)}$ are ruled. \label{Prop::PseudoIsoInvarRuledSurf}
\end{proposition}

\begin{example}[Spheres as invariant surfaces]\label{exe::PseudoSphrAsInvSurf}
Let $p\not=0$ be a constant and $\alpha(u)=(pu,0,\frac{pu^2}{2})$ be an isotropic circle. Applying a hyperbolic revolution to $\alpha$ gives
$$Zh_2(u,t)=(pu\cosh(t\phi),pu\sinh(t\phi),\frac{pu^2}{2}),$$
which parameterizes the sphere of parabolic type $\Sigma_1^2(p):z=\frac{1}{2p}(x^2-y^2)$. On the other hand, applying a parabolic revolution  with $c=0$, $a=pc_1$, and $b=-pc_2$,
$$Zh_3(u,t)=(at+pu,bt,\frac{(ac_1+bc_2)}{2}t^2+c_1tpu+\frac{pu^2}{2}),$$
which also parameterizes the sphere of parabolic type $\Sigma_1^2(p)$.

Similarly, the sphere of parabolic type $z=\frac{1}{2p}(y^2-x^2)$ can be obtained from both Lorentzian revolution and parabolic revolution (with $c=0$, $a=-pc_1$, and $b=pc_2$). See also example \ref{exe::SphrAsInvSurf} and the comment following it.
\end{example}

\section{Differential geometry of simply isotropic invariant surfaces}
\label{sec::DGSimplyInvSurf}

Now we compute the first and second fundamental forms of admissible invariant surfaces described in Subsect. \ref{subsec::SimplyInvSurf}, from which we derive the Gaussian and mean curvatures of simply isotropic helicoidal invariant surfaces in Subsect. \ref{subsec::DGSimplyInvHelSurf} and of invariant parabolic revolution  surfaces in Subsect. \ref{subsec::DGSimplyInvParaSurf}. In addition,  we solve both the prescribed Gaussian and mean  curvatures problems for helicoidal (Theorems \ref{thr::SimplyPrescKhelSurfiType} and \ref{thr::SimplyPrescHhelSurfiType}) and parabolic revolution (Theorems \ref{thr::SimplyPrescKParaSurfiType} and \ref{thr::SimplyPrescHParaSurfiType}) surfaces of i-type. On the other hand, for surfaces of ni-type we only solve the problem of prescribed Gaussian curvature for helicoidal surfaces (Theorem \ref{thr::SimplyPrescKhelSurfNiType}).

\subsection{Helicoidal and Euclidean revolution surfaces}
\label{subsec::DGSimplyInvHelSurf}

The first fundamental form of helicoidal surfaces of ni-type is given by Eq. (\ref{eq::SimplyIso1stFFhelicoidalSurfniType}):
$$\mathrm{I}=(x'\,^2+y'\,^2)\mathrm{d}u^2+2\phi(xy'-x'y)\mathrm{d}u\rmd t+\phi^2(x^2+y^2)\rmd t^2.$$

For the second fundamental form, we first compute
$$\partial_uY_2\times\partial_tY_2 = (cx'\sin(t\phi)+cy'\cos(t\phi),-cx'\cos(t\phi)+cy'\sin(t\phi),\phi(xx'+yy')).$$
Thus, the relative normal is  given by
\begin{equation}
N_h = \left(\frac{c}{\phi}\frac{x'\sin(t\phi)+y'\cos(t\phi)}{xx'+yy'},\frac{c}{\phi}\frac{-x'\cos(t\phi)+y'\sin(t\phi)}{xx'+yy'},1\right).
\end{equation}

Taking the Euclidean inner product of $N_h$ with the second derivatives of $Y_2$,
$$\partial_u^2Y_2 = (x''\cos(t\phi)-y''\sin(t\phi),x''\sin(t\phi)+y''\cos(t\phi),0),
$$
$$
\partial_{ut}^2Y_2 = (-\phi x'\sin(t\phi)-\phi y'\cos(t\phi),\phi x'\cos(t\phi)-\phi y'\sin(t\phi),0),
$$
and
$$\partial_t^2Y_2 = (-\phi^2 x\cos(t\phi)+\phi^2 y\sin(t\phi),-\phi^2x\sin(t\phi)-\phi^2y\cos(t\phi),0),
$$
gives the second fundamental form
\begin{equation}
\mathrm{II}=\frac{c(x''y'-x'y'')\rmd u^2-2c\phi(x'^2+y'^2)\rmd u\rmd t+c\phi^2(x'y-xy')\rmd t^2}{\phi(xx'+yy')}.
\end{equation}

\begin{proposition}
The Gaussian and mean curvatures of an admissible helicoidal surface of ni-type are respectively 
\begin{equation}
K=\frac{c^2}{\phi^2}\frac{(x'y-xy')(x''y'-x'y'')-(x'\,^2+y'\,^2)^2}{(xx'+yy')^4}
\end{equation}
and
\begin{equation}
H = \frac{c}{\phi}\frac{(x^2+y^2)(x''y'-x'y'')+(x'\,^2+y'\,^2)(xy'-x'y)}{2(xx'+yy')^3}.
\end{equation}
\end{proposition}

On the other hand, the first fundamental form of a helicoidal surface of i-type is given by Eq. (\ref{eq::SimplyIso1stFFhelicoidalSurfiType}):
$$\mathrm{I}=x'\,^2\mathrm{d}u^2+\phi^2x^2\rmd t^2.$$

For the second fundamental form, we first compute
$$\partial_uZ_2\times\partial_tZ_2 = (cx'\sin(t\phi)-\phi xz'\cos(t\phi),-cx'\cos(t\phi)-\phi xz'\sin(t\phi),\phi xx').
$$
Thus, the relative normal is given by
\begin{equation}
N_h = \left(\frac{cx'\sin(t\phi)-\phi xz'\cos(t\phi)}{\phi xx'},\frac{-cx'\cos(t\phi)-\phi xz'\sin(t\phi)}{\phi xx'},1\right).
\end{equation}

Taking the Euclidean inner product of $N_h$ with the second derivatives of $Z_2$,
$$\partial_u^2Z_2 = (x''\cos(t\phi),x''\sin(t\phi),z''),
\,\partial_{ut}^2Z_2 = (-\phi x'\sin(t\phi),\phi x'\cos(t\phi),0),
$$
and
$$\partial_t^2Z_2 = (-\phi^2 x\cos(t\phi),-\phi^2x\sin(t\phi),0),
$$
gives the second fundamental form
\begin{equation}
\mathrm{II}=\frac{-x(x''z'-x'z'')\rmd u^2-2c\, x'\,^2\rmd u\rmd t+\phi^2 x^2z'\rmd t^2}{xx'}.
\end{equation}

\begin{proposition}
The Gaussian and mean curvatures of an admissible  helicoidal surface of i-type are respectively
\begin{equation}
K=-\frac{z'(x''z'-x'z'')}{xx'\,^4}-\frac{c^2}{\phi^2x^4}
\mbox{ and }
H = \frac{x'\,^2z'-x(x''z'-x'z'')}{2xx'\,^3}.
\end{equation}
\end{proposition}

Notice that by imposing the condition $c=0$  we obtain the Euclidean revolution surfaces. Revolution surfaces of ni-type are necessarily contained in the $xy$-plane and, then, $K=0=H$. On the other hand, for helicoidal surfaces of i-type the mean curvature does not depend on the value of $c$, $H_{c=0}=H_{c\not=0}$, only their Gaussian curvature depends on $c$. Thus, we have

\begin{corollary}
The Gaussian and mean curvatures of an admissible simply isotropic Euclidean revolution surface of i-type are respectively
\begin{equation}
K=-\frac{z'(x''z'-x'z'')}{xx'\,^4}\mbox{ and }
H = \frac{x'\,^2z'-x(x''z'-x'z'')}{2xx'\,^3}.
\end{equation}
\end{corollary}

\subsubsection{Simply isotropic helicoidal surfaces with prescribed curvature}

\begin{thm}\label{thr::SimplyPrescKhelSurfNiType}
Let $K(s)$ be  continuous. There exists a 2-parameter family of helicoidal surfaces of ni-type with $K$ as the Gaussian curvature and whose generating curves $\alpha_{k_0,k_1}(s)$, $s$ arc-length parameter, are implicitly given by
\begin{equation}
 x^2(s)+y^2(s) = k_0 +2\int_{s_0}^{s}\left(k_1+\frac{3\phi^2}{c^2}\int_{s_0}^vK(w)\,\rmd w\right)^{-\frac{1}{3}}\rmd v,
\end{equation}
where $k_i$ ($i=1,2$) is a constant depending on the values of $x,y,x',y'$ at $s=s_0$. 
\end{thm}
\begin{proof}
If $\alpha(s)$ is parameterized by arc-length $s$, then
$$x'^2+y'^2=1\Rightarrow x'x''=-y'y''.$$ Thus, the Gaussian curvature is written as
\begin{eqnarray}
K & = & \frac{c^2}{\phi^2}\frac{x'x''yy'-x'^2yy''-xx''y'^2+xx'y'y''-1}{(xx'+yy')^4}\nonumber\\
& = &  -\frac{c^2}{\phi^2}\frac{1+xx''+yy''}{(xx'+yy')^4}=\frac{c^2}{3\phi^2}\frac{\rmd}{\rmd s}(xx'+yy')^{-3}.\nonumber
\end{eqnarray}
Integration gives 
$$\frac{1}{2}(x^2+y^2)'=(xx'+yy')=\left(k_1+\frac{3\phi^2}{c^2}\int K\right)^{-\frac{1}{3}}.$$
Finally, integrating this last equation gives the desired result.
\end{proof}

On the other hand, for helicoidal surfaces of i-type we can solve both prescribed Gaussian and mean curvatures problems.

\begin{thm}\label{thr::SimplyPrescKhelSurfiType}
Let $K(s)$ be a continuous function. Then, there exists a 2-parameter family of helicoidal surfaces of i-type with $K$ as the Gaussian curvature and whose generating curves $\alpha_{k_0,k_1}(s)=(s,0,z(s))$, $s$ arc-length parameter, satisfy
\begin{equation}
 z(s) = k_0 +\int_{s_0}^{s}\left(k_1-\frac{c^2}{\phi^2v^2}+2\int_{s_0}^vwK(w)\,\rmd w\right)^{\frac{1}{2}}\rmd v,
\end{equation}
where $k_i$ ($i=1,2$) is a constant depending on the values of $z,z'$ at $s=s_0$. 
\end{thm}
\begin{thm}\label{thr::SimplyPrescHhelSurfiType}
Let $H(s)$ be a continuous function. Then, there exists a 2-parameter family of helicoidal surfaces of i-type with $H$ as the mean curvature and whose generating curves $\alpha_{h_0,h_1}(s)=(s,0,z(s))$, $s>0$ arc-length parameter, satisfy
\begin{equation}
 z(s) = h_0 +h_1\ln\, s+\int_{s_0}^{s}\left(\frac{2}{v}\int_{s_0}^vwH(w)\,\rmd w\right)\rmd v,
\end{equation}
where $h_i$ ($i=1,2$) is a constant depending on the values of $z,z'$ at $s=s_0$. 
\end{thm}
\begin{proof}[Proof of Theorem \ref{thr::SimplyPrescKhelSurfiType}]
Let $\alpha(s)=(s,0,z(s))$ be the generating curve of an i-type helicoidal surface 
(an arc-length parameterization necessarily implies that $\alpha$ must be a graph curve in the $xz$-plane).
 The Gaussian curvature is then
$$K = \frac{z'z''}{s}-\frac{c^2}{\phi^2s^4}=\frac{1}{s}\frac{\rmd }{\rmd s}\frac{z'\,^2}{2}-\frac{c^2}{\phi^2s^4}\Rightarrow z'\,^2=k_1-\frac{c^2}{\phi^2s^2}+\int_{s_0}^s2wK(w)\,\rmd w.$$
Finally, a new integration gives
the desired result.
\end{proof}

\begin{proof}[Proof of Theorem \ref{thr::SimplyPrescHhelSurfiType}]
Let $\alpha(s)=(s,0,z(s))$ be the generating curve of an i-type helicoidal surface.
 The mean curvature is then
$$H= \frac{z'}{2s}+\frac{z''}{2}= \frac{1}{2s}\frac{\rmd}{\rmd s}(sz'),$$ 
which leads to
$$sz'=h_1+2\int vH(v)\,\rmd v\Rightarrow z'=\frac{h_1}{s}+\frac{2}{s}\int_{s_0}^s wH(w)\,\rmd w .$$
A new integration then gives the desired result.
\end{proof}

\begin{example}[Flat helicoidal surfaces]
Setting $K(s)\equiv0$ in Theorem \ref{thr::SimplyPrescKhelSurfiType}, gives
$$z(s)=z_0+s\sqrt{k_1-\frac{c^2}{\phi^2s^2}}+\frac{c}{\phi}\tan^{-1}\left(\frac{c}{\phi s\sqrt{k_1-\frac{c^2}{\phi^2s^2}}}\right),$$
which allows us to recover Corollary 1 of  \cite{KaracanTJM2017} (see also Sect. 4 of \cite{AydinAKUJSE2016}). 
\end{example}
\begin{example}[Revolution surfaces with constant Gaussian curvature]\label{exe::EuclRevSurfKcte}
Setting $K(s)\equiv K_0\not=0$ and $c=0$ in Theorem \ref{thr::SimplyPrescKhelSurfiType}, gives
$$z(s)=z_1+\frac{s}{2}\sqrt{z_0+K_0s^2}+\frac{z_0}{2\sqrt{\vert K_0\vert}}F(s\,;z_0,K_0),
$$
where $z_0$ and $z_1$ are constants and we have defined
$$
F(s\,;z_0,K_0)=\left\{
\begin{array}{ccc}
\ln\left(\sqrt{K_0}\,s+\sqrt{z_0+K_0s^2}\right)&\mbox{ if }&K_0>0\\[4pt]
\sin^{-1}\left(\sqrt{\frac{-K_0}{z_0}}\,s\right)&\mbox{ if }&z_0,K_0<0\\[4pt]
\end{array}
\right..
$$
In the limit $K_0\to0^-$ ($z_0>0$), $\frac{1}{\sqrt{-K_0}}F\to\frac{s}{\sqrt{z_0}}$ and, therefore, flat revolution surfaces of i-type are obtained by Euclidean rotations of a non-isotropic line, i.e., they are necessarily part of a cone.
\end{example}
\begin{example}[CMC helicoidal surfaces]
Setting $H(s)\equiv H_0$ in Theorem \ref{thr::SimplyPrescHhelSurfiType}, gives
$$
z(s) = z_0 + z_1 \ln s + \frac{1}{2}H_0\,s^2,
$$
 where $z_0$ and $z_1$ are constants. This allows us to recover Corollary 2 of \cite{KaracanTJM2017}. In addition, setting $H_0=0$, minimal helicoidal surfaces are ``logarithmoids'' and we recover Theorem 12.1 of \cite{Sachs1990}, p. 231.
\end{example}

\subsection{Parabolic revolution and warped translation surfaces}
\label{subsec::DGSimplyInvParaSurf}

The first fundamental form of a parabolic revolution surface of ni-type is given by Eq. (\ref{eq::SimplyIso1stFFParaRevSurfniType}):
$$\mathrm{I}=(x'\,^2+y'\,^2)\mathrm{d}u^2+2(ax'+by')\mathrm{d}u\rmd t+(a^2+b^2)\rmd t^2.$$

For the second fundamental form, we first compute
\begin{eqnarray}
\partial_uY_3\times\partial_tY_3 &=& (cy'+c_1t(ay'-bx')+c_1xy'+c_2yy')\mathbf{e}_1+\nonumber\\
+(-cx'&+&c_2t(ay'-bx')-c_1xx'-c_2x'y)\mathbf{e}_2+(bx'-ay')\mathbf{e}_3.\nonumber
\end{eqnarray}
Thus, the relative normal is then given by
$$
N_h =(\frac{c_1t[ay'-bx']+[c+c_1x+c_2y]y'}{bx'-ay'},\frac{c_2t[ay'-bx']-[c+c_1x+c_2y]x'}{bx'-ay'},1).
$$

Taking the Euclidean inner product of $N_h$ with the second derivatives of $Y_3$,
$$\partial_u^2Y_3 = (x'',y'',c_1tx''+c_2ty''),
\,\partial_{tu}^2Y_3=(c_1x'+c_2y')\mathbf{e}_3,
$$
and
$$
\partial_t^2Y_3 = (ac_1+bc_2)\mathbf{e}_3,
$$
gives the second fundamental form
\begin{equation}
\mathrm{II}=(c+c_1x+c_2y)\frac{(x''y'-x'y'')}{bx'-ay'}\rmd u^2+2(c_1x'+c_2y')\rmd u\rmd t+(ac_1+bc_2)\rmd t^2.
\end{equation}

\begin{proposition}
The Gaussian and mean curvatures of an admissible parabolic revolution surface of ni-type are respectively
\begin{equation}
K=\frac{(ac_1+bc_2)(c+c_1x+c_2y)(x''y'-x'y'')}{(bx'-ay')^3}-\left(\frac{c_1x'+c_2y'}{bx'-ay'}\right)^2
\end{equation}
and
\begin{equation}
H = \frac{(a^2+b^2)(c+c_1x+c_2y)(x''y'-x'y'')}{2(bx'-ay')^3}-
\end{equation}
$$-\frac{(ax'+by')(c_1x'+c_2y')}{(bx'-ay')^2}+\frac{(x'\,^2+y'\,^2)(ac_1+bc_2)}{2(bx'-ay')^2}.$$
\end{proposition}

On the other hand, the first fundamental form of a parabolic revolution surface of i-type is given by Eq. (\ref{eq::SimplyIso1stFFParaRevSurfiType}):
$$\mathrm{I}=x'\,^2\mathrm{d}u^2+2ax'\mathrm{d}u\rmd t+(a^2+b^2)\rmd t^2.$$

For the second fundamental form, we first compute
$$\partial_uZ_3\times\partial_tZ_3 = (-bc_1tx'-bz',az'-cx'-bc_2tx'-c_1xx',bx').
$$
Thus, the relative normal is then given by
\begin{equation}
N_h = \left(\frac{-bc_1tx'-bz'}{bx'},\frac{az'-cx'-bc_2tx'-c_1xx'}{bx'},1\right).
\end{equation}

Taking the Euclidean inner product of $N_h$ with the second derivatives of $Z_3$,
$$\partial_u^2Z_3 = (x'',0,c_1tx''+z''),\,
\partial_{tu}^2Z_3=(c_1x')\mathbf{e}_3,\mbox{ and }\,\partial_t^2Z_3 = (ac_1+bc_2)\mathbf{e}_3,
$$
gives the second fundamental form
\begin{equation}
\mathrm{II}=-\frac{x''z'-x'z''}{x'}\rmd u^2+2c_1x'\rmd u\rmd t+(ac_1+bc_2)\rmd t^2.
\end{equation}

\begin{proposition}
The Gaussian and mean curvatures of an admissible parabolic revolution surface of i-type are respectively
\begin{equation}
K=-\frac{[ac_1+bc_2][x''z'-x'z'']}{b^2x'\,^3}-\frac{c_1^2}{b^2},
H = \frac{bc_2-ac_1}{2b^2}-\frac{[a^2+b^2]}{2b^2}\frac{[x''z'-x'z'']}{x'\,^3}.
\end{equation}
\end{proposition}

Warped translation surfaces are obtained from the parabolic revolution ones by setting $c=0$ and $ac_1+bc_2=0$, but  $(c_1,c_2)\not=(0,0)$. Then,

\begin{corollary}
The Gaussian and mean curvatures of an admissible simply isotropic warped translation surface of  ni-type are respectively
\begin{equation}
K=-\left(\frac{c_1x'+c_2y'}{bx'-ay'}\right)^2
\end{equation}
and
\begin{equation}
H = \frac{(a^2+b^2)(c_1x+c_2y)(x''y'-x'y'')}{2(bx'-ay')^3}-\frac{(ax'+by')(c_1x'+c_2y')}{(bx'-ay')^2};
\end{equation}
for the i-type they are respectively
\begin{equation}
K=-\left(\frac{c_1}{b}\right)^2
\mbox{ and }
H = -\frac{(a^2+b^2)}{2b^2}\frac{(x''z'-x'z'')}{x'\,^3}+\frac{bc_2-ac_1}{2b^2}.
\end{equation}
\end{corollary}

\subsubsection{Simply isotropic parabolic revolution surfaces with prescribed curvature}

\begin{thm}\label{thr::SimplyPrescKParaSurfiType}
Let $K(s)$ be  continuous. There exists a 2-parameter family of parabolic revolution surfaces of i-type with $K$ as the Gaussian curvature and whose generating curves $\alpha_{k_0,k_1}(s)=(s,0,z(s))$, $s$ arc-length parameter, satisfy
\begin{equation}
 z(s) = k_0 +k_1s+\frac{c_1^2\,s^2}{2(ac_1+bc_2)}+\frac{b^2}{ac_1+bc_2}\int_{s_0}^{s}\int_{s_0}^vK(w)\,\rmd w\,\rmd v,
\end{equation}
where $k_i$ ($i=1,2$) is a constant depending on the values of $z,z'$ at $s=s_0$. 
\end{thm}
\begin{thm}\label{thr::SimplyPrescHParaSurfiType}
Let $H(s)$ be a continuous function. Then, there exists a 2-parameter family of parabolic revolution surfaces of i-type with $H$ as the mean curvature and whose generating curves $\alpha_{h_0,h_1}(s)=(s,0,z(s))$, $s>0$ arc-length parameter, satisfy
\begin{equation}
 z(s) = h_0 +h_1s+\frac{ac_1-bc_2}{2(a^2+b^2)}\,s^2+\frac{2b^2}{a^2+b^2}\int_{s_0}^{s}\int_{s_0}^vH(w)\,\rmd w\,\rmd v,
\end{equation}
where $h_i$ ($i=1,2$) is a constant depending on the values of $z,z'$ at $s=s_0$. 
\end{thm}
\begin{proof}[Proof of Theorem \ref{thr::SimplyPrescKParaSurfiType}]
Let $\alpha(s)=(s,0,z(s))$ be the generating curve of an i-type parabolic revolution surface 
(an arc-length parameterization necessarily implies that $\alpha$ must be a graph curve in the $xz$-plane).
 The Gaussian curvature is then
$$K = \frac{ac_1+bc_2}{b^2}z''-\frac{c_1^2}{b^2}\Rightarrow z''=\frac{c_1^2}{ac_1+bc_2}+\frac{b^2}{ac_1+bc_2}K.$$ 
Finally, integrating twice gives
the desired result.
\end{proof}

\begin{proof}[Proof of Theorem \ref{thr::SimplyPrescHParaSurfiType}]
Let $\alpha(s)=(s,0,z(s))$ be the generating curve of an i-type parabolic revolution surface.
 The mean curvature is then
$$H = \frac{a^2+b^2}{2b^2}z''+\frac{bc_2-ac_1}{2b^2}\Rightarrow z''=\frac{ac_1-bc_2}{a^2+b^2}+\frac{2b^2}{a^2+b^2}H.$$ 
Finally, integrating twice gives
the desired result.
\end{proof}

\begin{example}[Constant curvature parabolic revolution surfaces]
Setting $K(s)=K_0$ in Theorem \ref{thr::SimplyPrescKParaSurfiType} gives, for arbitrary constants $z_0$ and $z_1$,
$$z(s)=z_0+z_1s+\frac{c_1^2+b^2K_0}{2(ac_1+bc_2)}\,s^2.$$
On the other hand, setting $H(s)=H_0$ in Theorem \ref{thr::SimplyPrescHParaSurfiType} gives
$$z(s)=z_0+z_1s+\frac{ac_1-bc_2+2b^2H_0}{2(a^2+b^2)}\,s^2.$$
\end{example}

\section{Differential geometry of pseudo-isotropic invariant surfaces}
\label{sec::DGPseudoInvSurf}

Here we compute the first and second fundamental forms of the pseudo-isotropic invariant surfaces described in Subsect. \ref{subsec::PseudoInvSurf} and derive the Gaussian and mean curvatures of  helicoidal and parabolic revolution surfaces in Subsects. \ref{subsec::DGPseudoInvHelSurf} and  \ref{subsec::DGPseudoInvParaSurf}, respectively. In addition,  we solve both the prescribed Gaussian and mean curvatures problems for helicoidal (Theorems \ref{thr::PseudoPrescKhelSurfiType} and \ref{thr::PseudoPrescHhelSurfiType}) and parabolic revolution (Theorems \ref{thr::PseudoPrescKParaSurfiType} and \ref{thr::PseudoPrescHParaSurfiType}) surfaces of i-type. Finally, for surfaces of ni-type we only solve the problem of prescribed Gaussian curvature for helicoidal surfaces (Theorem \ref{thr::PseudoPrescKhelSurfNiType}).

\subsection{Helicoidal and Lorentzian revolution surfaces}
\label{subsec::DGPseudoInvHelSurf}

The first fundamental form of a hyperbolic helicoidal surface of ni-type is given by Eq. (\ref{eq::PseudoIso1stFFhelicoidalSurfniType}):
$$\mathrm{I}=(x'\,^2-y'\,^2)\mathrm{d}u^2+2\phi(x'y-xy')\mathrm{d}u\rmd t-\phi^2(x^2-y^2)\rmd t^2.$$

The pseudo-isotropic relative normal is given by
\begin{equation}
N_h = \frac{c}{\phi}\left(\frac{x'\sinh(t\phi)+y'\cosh(t\phi)}{xx'-yy'},\frac{x'\cosh(t\phi)+y'\sinh(t\phi)}{xx'-yy'},\frac{\phi}{c}\right).
\end{equation}
The second fundamental form is
\begin{equation}
\mathrm{II}=\frac{c(x''y'-x'y'')\rmd u^2-2c\phi(x'^2-y'^2)\rmd u\rmd t+c\phi^2(xy'-x'y)\rmd t^2}{\phi(xx'-yy')}.
\end{equation}

\begin{proposition}
The Gaussian and mean curvatures of an admissible and pseudo-isotropic helicoidal surface of ni-type are respectively
\begin{equation}
K = \frac{c^2}{\phi^2}\frac{(x'\,^2-y'\,^2)^2-(xy'-x'y)(x''y'-x'y'')}{(xx'-yy')^4}.
\end{equation}
and
\begin{equation}
H = \frac{c}{\phi}\frac{(x^2-y^2)(x''y'-x'y'')+(x'\,^2-y'\,^2)(xy'-x'y)}{2(xx'-yy')^3}.
\end{equation}
\end{proposition}

On the other hand, the first fundamental form of a helicoidal surface of i-type is given by Eqs. (\ref{eq::ZhPseudoIso1stFFhelicoidalSurfiType}) and (\ref{eq::WhPseudoIso1stFFhelicoidalSurfiType}):
$$
\left\{
\begin{array}{lc}
\mathrm{I}=x'\,^2\mathrm{d}u^2-\phi^2x^2\rmd t^2 &,\mbox{ if }
\alpha=(x,0,z)\\
\mathrm{I}=-y'\,^2\mathrm{d}u^2+\phi^2y^2\rmd t^2&,\mbox{ if }
\alpha=(0,y,z)\\
\end{array}
\right..
$$

The relative normal is  given by
\begin{equation}
N_h = \left(\frac{cx'\sinh(t\phi)-c\phi xz'\cosh(t\phi)}{\phi xx'},\frac{cx'\cosh(t\phi)-c\phi xz'\sinh(t\phi)}{\phi xx'},1\right).
\end{equation}
The second fundamental form is
\begin{equation}
    \left\{
\begin{array}{lc}
\mathrm{II}=-\displaystyle\frac{x(x''z'-x'z'')\rmd u^2+2c x'\,^2\rmd u\rmd t+\phi^2 x^2z'\rmd t^2}{xx'}&,\mbox{ if }
\alpha=(x,0,z)\\[5pt]
\mathrm{II}=-\displaystyle\frac{y(y''z'-y'z'')\rmd u^2+2c y'\,^2\rmd u\rmd t+\phi^2 y^2z'\rmd t^2}{yy'}&,\mbox{ if }
\alpha=(0,y,z)\\
\end{array}
\right..
\end{equation}

\begin{proposition}
The Gaussian and mean curvatures of an admissible and pseudo-isotropic helicoidal surface of i-type are respectively
\begin{equation}
K = \frac{c^2}{\phi^2f^4}-\frac{z'}{ff'\,^4}(f''z'-f'z'')\mbox{ and }H=\varepsilon\,\frac{f'\,^2z'-f(f''z'-f'z'')}{2ff'\,^3},
\end{equation}
where $\varepsilon=+1$
if the generating curve is spacelike, $\alpha(u)=(f(u),0,z(u))$, and $\varepsilon=-1$ if the generating curve is timelike, $\alpha(u)=(0,f(u),z(u))$.
\end{proposition}

Notice that if $c=0$  we obtain the Lorentzian revolution surfaces. Revolution surfaces of ni-type are  contained in the $xy$-plane and then $K=0=H$. For helicoidal surfaces of i-type $H$ does not depend on $c$: $H_{c=0}=H_{c\not=0}$. (Only $K$ depends on $c$.)

\begin{corollary}
The Gaussian and mean curvatures of an admissible pseudo-isotropic revolution surface of i-type are respectively
\begin{equation}
K = -\frac{z'}{ff'\,^4}(f''z'-f'z'')\mbox{ and }H=\varepsilon\,\frac{f'\,^2z'-f(f''z'-f'z'')}{2ff'\,^3},
\end{equation}
where $\varepsilon=+1$
if the generating curve is spacelike, $\alpha(u)=(f(u),0,z(u))$, and $\varepsilon=-1$ if the generating curve is timelike, $\alpha(u)=(0,f(u),z(u))$.
\end{corollary}

\subsubsection{Pseudo-isotropic helicoidal surfaces with prescribed curvature}

\begin{thm}\label{thr::PseudoPrescKhelSurfNiType}
Let $K(s)$ be continuous. There exists a 2-parameter family of helicoidal surfaces of ni-type with $K$ as the Gaussian curvature and whose generating curves $\alpha_{k_0,k_1}(s)$, $s$ arc-length parameter, are implicitly given by
\begin{equation}
 x^2(s)-y^2(s) = k_0 +2\int_{s_0}^{s}\left(k_1-\frac{3\varepsilon\phi^2}{c^2}\int_{s_0}^vK(w)\,\rmd w\right)^{-\frac{1}{3}}\rmd v,
\end{equation}
where $k_i$ ($i=1,2$) is a constant depending on the values of $x,y,x',y'$ at $s=s_0$ and $\varepsilon=x'\,^2-y'\,^2=\pm1$ determines the causal character of $\alpha_{k_0,k_1}(s)$. 
\end{thm}
\begin{proof}
If $\alpha(s)$ is parameterized by arc-length $s$, then
$$x'\,^2-y'\,^2=\varepsilon\in\{-1,+1\}\Rightarrow x'x''=y'y''.$$ Therefore, the Gaussian curvature is 
\begin{eqnarray}
K & = & \frac{c^2}{\phi^2}\frac{1-xx''y'^2+xx'y'y''+x'x''yy'-yy''x'^2}{(xx'-yy')^4}\nonumber\\
& = &  \varepsilon\frac{c^2}{\phi^2}\frac{\varepsilon+xx''-yy''}{(xx'-yy')^4}=-\varepsilon\frac{c^2}{3\phi^2}\frac{\rmd}{\rmd s}(xx'-yy')^{-3}.\nonumber
\end{eqnarray}
Integration gives 
$$\frac{1}{2}(x^2-y^2)'=(xx'-yy')=\left(k_1-\frac{3\varepsilon\phi^2}{c^2}\int K\right)^{-\frac{1}{3}}.$$
Finally, integrating this last equation gives the desired result.
\end{proof}

As in $\mathbb{I}^3$, for pseudo-isotropic helicoidal surfaces of i-type we solve both Gaussian and mean prescribed curvatures problems. (See proofs of Theorems \ref{thr::SimplyPrescKhelSurfiType} and \ref{thr::SimplyPrescHhelSurfiType}.)

\begin{thm}\label{thr::PseudoPrescKhelSurfiType}
Let $K(s)$ be a continuous function. Then, there exists a 2-parameter family of pseudo-isotropic helicoidal surfaces of i-type with $K$ as the Gaussian curvature and whose generating curves $\alpha_{k_0,k_1}(s)=(s,0,z(s))$, or $\alpha_{k_0,k_1}(s)=(0,s,z(s))$, where $s>0$ is an arc-length parameter, satisfy
\begin{equation}
 z(s) = k_0 +\int_{s_0}^{s}\left(k_1+\frac{c^2}{\phi^2v^2}+2\int_{s_0}^vwK(w)\,\rmd w\right)^{\frac{1}{2}}\rmd v,
\end{equation}
where $k_i$ ($i=1,2$) is a constant depending on the values of $z,z'$ at $s=s_0$. 
\end{thm}
\begin{thm}\label{thr::PseudoPrescHhelSurfiType}
Let $H(s)$ be a continuous function. Then, there exists a 2-parameter family of pseudo-isotropic helicoidal surfaces of i-type with $H$ as the mean curvature and whose generating curves $\alpha_{h_0,h_1}(s)=(s,0,z(s))$, or $\alpha_{h_0,h_1}(s)=(0,s,z(s))$, where $s>0$ is an arc-length parameter, satisfy
\begin{equation}
 z(s) = h_0 +h_1\ln\, s+\varepsilon\int_{s_0}^{s}\left(\frac{2}{v}\int_{s_0}^vwH(w)\,\rmd w\right)\rmd v,
\end{equation}
where $h_i$ ($i=1,2$) is a constant depending on the values of $z,z'$ at $s=s_0$ and $\varepsilon=\pm1$ determines the causal character of the generating curves. 
\end{thm}

\begin{example}[Flat helicoidal surfaces]
Setting $K(s)\equiv0$ in Theorem \ref{thr::PseudoPrescKhelSurfiType}, gives
$$z(s)=z_0+s\sqrt{z_1+\frac{c^2}{\phi^2s^2}}-\frac{c}{\phi}\ln\left(\frac{c}{\phi s}+\sqrt{z_1+\frac{c^2}{\phi^2s^2}}\right),$$
where $z_0$ and $z_1$ are arbitrary constants.
\end{example}
\begin{example}[Revolution surfaces with constant Gaussian curvature]
Setting $K(s)\equiv K_0$ and $c=0$ in Theorem \ref{thr::PseudoPrescKhelSurfiType} gives the same equation as in $\mathbb{I}^3$. Thus, $z(s)$ is given in example \ref{exe::EuclRevSurfKcte}.
(For $z_0,K_0>0$ we recover Theorem 5.1 in \cite{AydinTJM2018}.)
\end{example}
\begin{example}[CMC helicoidal surfaces]
Setting $H(s)\equiv H_0$ in Theorem \ref{thr::PseudoPrescHhelSurfiType}, gives
$$
z(s) = z_0 + z_1 \ln s + \frac{\varepsilon}{2} H_0\,s^2,
$$
where $z_0$ and $z_1$ are constants. All minimal helicoidal surfaces are ``logarithmoids''. If $\varepsilon=+1$, i.e., $\alpha(s)=(s,0,z(s))$, we recover Theorem 5.2 in \cite{AydinTJM2018}. 
\end{example}

\subsection{Parabolic revolution and warped translation pseudo-isotropic surfaces}
\label{subsec::DGPseudoInvParaSurf}

The first fundamental form of a parabolic revolution surface of ni-type is given by Eq. (\ref{eq::PseudoIso1stFFParaRevSurfniType}):
$$\mathrm{I}=(x'\,^2-y'\,^2)\mathrm{d}u^2+2(ax'-by')\mathrm{d}u\rmd t+(a^2-b^2)\rmd t^2.$$

The relative normal is given by
$$
N_h =(\frac{[c+c_1x+c_2y]y'+c_1t[ay'-bx']}{bx'-ay'},\frac{[c+c_1x+c_2y]x'-c_2t[ay'-bx']}{bx'-ay'},1).
$$
The second fundamental form is
\begin{equation}
\mathrm{II} = \frac{c+c_1x+c_2y}{bx'-ay'}(x''y'-x'y'')\mathrm{d}u^2+2(c_1x'+c_2y')\mathrm{d}u\mathrm{d}t+(ac_1+bc_2)\mathrm{d}t^2.
\end{equation}

\begin{proposition}
The Gaussian and mean curvatures of an admissible parabolic revolution surface of ni-type are respectively
\begin{equation}
K = \left(\frac{c_1x'+c_2y'}{bx'-ay'}\right)^2-\frac{(ac_1+bc_2)(c+c_1x+c_2y)}{(bx'-ay')^3}(x''y'-x'y'')
\end{equation}
and
\begin{equation}
H = -\frac{(a^2-b^2)(c+c_1x+c_2y)(x''y'-x'y'')}{2(bx'-ay')^3}+
\end{equation}
$$+\frac{(ax'-by')(c_1x'+c_2y')}{(bx'-ay')^2}-\frac{(x'\,^2-y'\,^2)(ac_1+bc_2)}{2(bx'-ay')^2}.$$
\end{proposition}

On the other hand, the first fundamental form of a parabolic revolution surface of i-type is given by Eqs. (\ref{eq::ZhPseudoIso1stFFParaRevSurfiType}) and (\ref{eq::WhPseudoIso1stFFParaRevSurfiType}):
$$
\left\{\begin{array}{lc}
\mathrm{I}=x'\,^2\mathrm{d}u^2+2ax'\mathrm{d}u\rmd t+(a^2-b^2)\rmd t^2&,\mbox{ if } \alpha=(x,0,z)\\
\mathrm{I}=-y'\,^2\mathrm{d}u^2-2by'\rmd u\rmd t+(a^2-b^2)\rmd t^2&,\mbox{ if }\alpha=(0,y,z)
\end{array}
\right..
$$
The relative normal is given by
\begin{equation}
N_h = \left(\frac{-bc_1tx'-bz'}{bx'},\frac{cx'-az'+bc_2tx'+c_1xx'}{bx'},1\right).
\end{equation}
The second fundamental form is
\begin{equation}
    \left\{\begin{array}{lc}
\mathrm{II}=-\frac{x''z'-x'z''}{x'}\rmd u^2+2c_1x'\rmd u\rmd t+(ac_1+bc_2)\rmd t^2&,\mbox{ if } \alpha=(x,0,z)\\[5pt]
\mathrm{II}=-\frac{y''z'-y'z''}{y'}\rmd u^2+2c_2y'\rmd u\rmd t+(ac_1+bc_2)\rmd t^2&,\mbox{ if }\alpha=(0,y,z)
\end{array}
\right..
\end{equation}

\begin{proposition}
The Gaussian and mean curvatures of an admissible and pseudo-isotropic parabolic revolution surface of i-type are respectively 
\begin{equation}
K=\frac{(ac_1+bc_2)(f''z'-f'z'')}{B^2f'\,^3}+\frac{C^2}{B^2},
H = \frac{(a^2-b^2)}{2B^2}\frac{(f''z'-f'z'')}{f'\,^3}-\frac{bc_2-ac_1}{2B^2},
\end{equation}
where $C=c_1$ and $B=b$, if the generating curve is $\alpha(u)=(f(u),0,z(u))$ ($\alpha$ spacelike), and $C=c_2$ and $B=a$, if  $\alpha(u)=(0,f(u),z(u))$ ($\alpha$ timelike).
\end{proposition}

\begin{corollary}
The Gaussian and mean curvatures of an admissible pseudo-isotropic warped translation surface of ni-type are 
\begin{equation}
K = \left(\frac{c_1x'+c_2y'}{bx'-ay'}\right)^2,
\end{equation}
and
\begin{equation}
H = -\frac{(a^2-b^2)(c_1x+c_2y)(x''y'-x'y'')}{2(bx'-ay')^3}+\frac{(ax'-by')(c_1x'+c_2y')}{(bx'-ay')^2},
\end{equation}
respectively; while for the i-type they are respectively
\begin{equation}
K=\left(\frac{C}{B}\right)^2
\mbox{ and }
H = -\frac{bc_2-ac_1}{2B^2}+\frac{(a^2-b^2)}{2B^2}\frac{(f''z'-f'z'')}{f'\,^3},
\end{equation}
where $C=c_1$ and $B=b$, if the generating curve is $\alpha(u)=(f(u),0,z(u))$ ($\alpha$ spacelike), and $C=c_2$ and $B=a$, if  $\alpha(u)=(0,f(u),z(u))$ ($\alpha$ timelike).
\end{corollary}

\subsubsection{Pseudo-isotropic parabolic revolution surfaces with prescribed curvature}

As in $\mathbb{I}^3$, for parabolic revolution surfaces of i-type we can solve both  prescribed Gaussian and mean curvatures problems. The proofs of the two theorems below are completely analogous to those of Theorems \ref{thr::SimplyPrescKParaSurfiType} and \ref{thr::SimplyPrescHParaSurfiType}, respectively.

\begin{thm}\label{thr::PseudoPrescKParaSurfiType}
Let $K(s)$ be continuous. There exists a 2-parameter family of pseudo-isotropic parabolic revolution surfaces of i-type with $K$ as the Gaussian curvature and whose generating curves $\alpha_{k_0,k_1}(s)$,   $s$ is an arc-length parameter, satisfy
\begin{equation}
 z(s) = k_0 +k_1s+\frac{c_1^2\,s^2}{2(ac_1+bc_2)}-\frac{B^2}{ac_1+bc_2}\int_{s_0}^{s}\int_{s_0}^vK(w)\,\rmd w\,\rmd v,
\end{equation}
where $k_i$ ($i=1,2$) is a constant depending on the values of $z,z'$ at $s=s_0$, $B=b$ if $\alpha_{h_0,h_1}(s)=(s,0,z(s))$, and $B=a$ if $\alpha_{h_0,h_1}(s)=(0,s,z(s))$.
\end{thm}
\begin{thm}\label{thr::PseudoPrescHParaSurfiType}
Let $H(s)$ be continuous. There exists a 2-parameter family of pseudo-isotropic parabolic revolution surfaces of i-type with $H$ as the mean curvature and whose generating curves $\alpha_{h_0,h_1}(s)$, $s>0$ is an arc-length parameter,  satisfy
\begin{equation}
 z(s) = h_0 +h_1s+\frac{ac_1-bc_2}{2(a^2-b^2)}\,s^2-\frac{2B^2}{a^2-b^2}\int_{s_0}^{s}\int_{s_0}^vH(w)\,\rmd w\,\rmd v,
\end{equation}
where $h_i$ ($i=1,2$) is a constant depending on the values of $z,z'$ at $s=s_0$, $B=b$ if $\alpha_{h_0,h_1}(s)=(s,0,z(s))$, and $B=a$ if $\alpha_{h_0,h_1}(s)=(0,s,z(s))$. 
\end{thm}

\begin{example}[Constant curvature parabolic revolution surfaces]
Setting $K(s)=K_0$ in Theorem \ref{thr::PseudoPrescKParaSurfiType} or $H(s)=H_0$ in Theorem \ref{thr::PseudoPrescHParaSurfiType} gives, for some $z_0$ and $z_1$, surfaces with constant Gaussian or mean curvatures with generating curve respectively given by
$$z(s)=z_0+z_1s+\frac{c_1^2-B^2K_0}{2(ac_1+bc_2)}\,s^2\mbox{ or }z(s)=z_0+z_1s+\frac{ac_1-bc_2-2B^2H_0}{2(a^2-b^2)}\,s^2.$$
\end{example}

\section*{Acknowledgements} 
The Author would like to thank Muhittin E. Aydin and Rafael L\'opez for useful discussions  and also the Staff from the Departamento de Matem\'atica at the Universidade Federal de Pernambuco (Recife, Brazil) where part of this research was done while the Author worked as a temporary lecturer.

\end{document}